\documentclass{article}
\usepackage{geometry,amsmath,amssymb,enumerate,amsfonts,epsfig,amsthm,bbm}
\newcommand{\R}{\mathbb{R}}
\newcommand{\Z}{\mathbb{Z}}
\newcommand{\Prob}{\mathbb{P}}
\newcommand{\E}{\mathbb{E}}
\newcommand{\1}{\mathbbm{1}}

\newtheorem{cor}{Corollary}
\newtheorem{lem}{Lemma}
\newtheorem{thm}{Theorem}

\newtheorem{clm}{Claim}
\newtheorem{prop}{Proposition}
\theoremstyle{definition}
\newtheorem{rmk}{Remark}

\begin{document}
\begin{center}
{\Large\textbf{Superdiffusive and Subdiffusive Exceptional Times in The Dynamical Discrete Web}} \vspace{2mm}\\
\text{Accepted for publication in Stoch. Proc.
Applic.}\vspace{4mm}\\
{\large\textbf{Dan Jenkins }}\\
\text{Courant Institute of Mathematical Sciences, NYU, New York, NY 10012}
\end{center}

\begin{abstract}
The dynamical discrete web (DyDW) is a system of one-dimensional coalescing random walks that evolves in an extra dynamical time parameter, $\tau$.  At any deterministic $\tau$ the paths behave as coalescing simple symmetric random walks.  It has been shown in \cite{OP} that there exist exceptional dynamical times, $\tau$, at which the path from the origin, $S_0^{\tau}$, is $K$-subdiffusive, meaning $S_0^{\tau}(t) \leq j+K\sqrt{t}$ for all $t$, where $t$ is the random walk time, and $j$ is some constant.  In this paper we consider for the first time the existence of superdiffusive exceptional times.  To be specific, we consider $\tau$ such that $\limsup_{t \to \infty} S_0^{\tau}(t)/\sqrt{t\log(t)} \geq C$.  We show that such exceptional times exist for small values of $C$, but they do not exist for large $C$.  The other goal of this paper is to establish the existence of exceptional times for which the path from the origin is $K$-subdiffusive in both directions, i.e. $\tau$ such that $|S_0^{\tau}(t)| \leq j+K\sqrt{t}$ for all $t$.  We also obtain upper and lower bounds for the Hausdorff dimensions of these two-sided subdiffusive exceptional times.  For the superdiffusive exceptional times we are able to get a lower bound on Hausdorff dimension but not an upper bound.

\end{abstract}

\begin{section}{Introduction}$\quad$  
This paper examines the dynamical discrete web (DyDW), a system of coalescing random walks that evolves in a continuous dynamical time parameter.  The dynamical discrete web was introduced by Howitt and Warren in \cite{OP1}.  The DyDW and related systems have been considered as models for erosion and other hydrological phenomena (see \cite{OP5},\cite{OP6}).  We examine ``exceptional times" for the DyDW.  These are dynamical times at which paths from the DyDW display behavior that would have probability zero for a standard random walk, or for the DyDW observed at a deterministic time.


\begin{figure}[b!!!]
\begin{center}
 \includegraphics[scale=0.5]{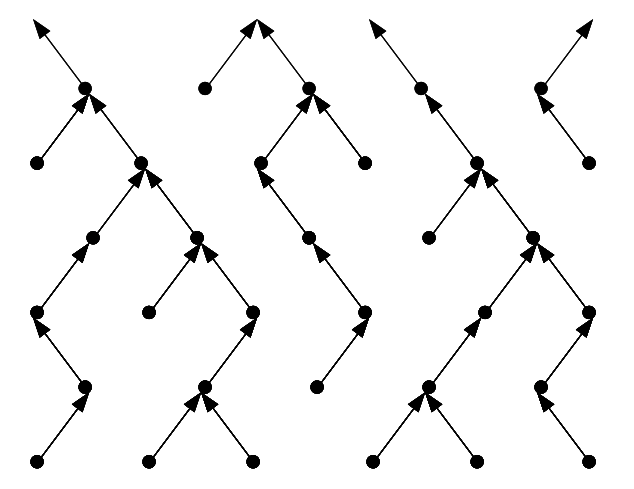}
\end{center}
 \caption{A partial realization of the discrete web.  Each arrow independently points left or right with probability $1/2$.  In the dynamical discrete web, each arrow has an independent Poisson clock and resets whenever it rings.}
 \label{fig: 1}
\end{figure}

First we define the dynamical discrete web, and briefly describe our main results.  This paper follows \cite{OP} closely; for a more thorough introduction to the subject see Section 1 of that paper. 

To discuss the DyDW, we first define the discrete web (DW).  The discrete web is a system of coalescing one-dimensional simple symmetric random walks.  To construct it, we independently assign to each point in $\Z^2_{\text{even}}:=\{(x,t)\in \Z^2 : \; x+t \text{ is even} \}$ a symmetric, $\pm 1$-valued Bernoulli random variable, $\xi_{(x,t)}$.  We then draw an arrow from $(x,t)$ to $(x+\xi_{(x,t)}, t+1)$ (see Figure \ref{fig: 1}).  For each $(x,t) \in \Z^2_{\text{even}}$, we let $S_{(x,t)}(\cdot)$ be the path that starts at $(x,t)$ and follows the arrows from there.  The discrete web is the collection of all such paths for $(x,t) \in \Z^2_{\text{even}}$.  As the figures and the ordering of $(x,t)$ suggest, we let the path time coordinate, $t$, run vertically, and the space coordinate, $x$, run horizontally.  Future references to left/right or up/down in $\Z^2_{\text{even}}$ should be understood according to this convention.

The DyDW was first introduced by Howitt and Warren in \cite{OP1}.  It is a discrete web that evolves in an extra dynamical time parameter, $\tau$, by letting the arrows independently switch directions as $\tau$ increases.  To accomplish this, we assign to each $(x,t) \in \Z^2_{\text{even}}$ an independent, rate one Poisson clock.  When the clock at $(x,t)$ rings, we reset the arrow at $(x,t)$ by replacing it with a new, independent arrow (which may or may not agree with the previous arrow).  Note that this gives the same distribution as if we had forced the arrows to switch at half the rate.  These dynamics correspond to replacing the $\xi_{(x,t)}$ from the DW with right-continuous $\tau$-varying versions, $\xi_{(x,t)}^{\tau}$.  We then let $\mathcal{W}(\tau)$ denote the discrete web constructed from the $\xi_{(x,t)}^{\tau}$'s, and let $S_{(x,t)}^{\tau}(\cdot)$ denote the path from $\mathcal{W}(\tau)$ starting at $(x,t)$.  Note that at any deterministic $\tau$, $\mathcal{W}(\tau)$ is distributed as a discrete web, so all paths in $\mathcal{W}(\tau)$ behave as simple symmetric random walks.

An ``exceptional time'' refers to a random dynamical time at which the DyDW behaves in a way that would have probability zero for the DW.  The study of exceptional times for the DyDW has been motivated by earlier work on dynamical percolation, see \cite{OP2}, \cite{OP3}.  Similarly to the DyDW, dynamical percolation consists of a lattice of Bernoulli random variables which reset according to independent Poisson processes.  For static (non-dynamical) percolation with critical edge probabilities it is believed that no infinite cluster should exist.  This is proven for dimension two and large dimensions (see \cite{OP11}, for example).  In \cite{OP3} it was shown that critical two-dimensional dynamical percolation has exceptional times where this fails, i.e. where an infinite cluster exists.  However, no such exceptional times exist for large dimensions, see \cite{OP2}.

Exceptional times for the DyDW were first studied by Fontes, Newman, Ravishankar and Schertzer in \cite{OP}.  They use techniques similar to those used for dynamical percolation to show that there exist exceptional times for the DyDW.  Their paper shows the existence of $\tau$ at which the path from the origin ($S_0^{\tau}$) is subdiffusive in one direction, growing slower than allowed by the classical law of the iterated logarithm.  To be specific, they show that for sufficiently large $K,j$:
\vspace*{-1mm}\begin{align} \label{sub+}
\Prob\left(\exists \tau \in [0,1] \text{ s.t. } S_0^{\tau}(t) \leq j+K\sqrt{t} \text{   for all }t\geq 0\right)>0.
\vspace*{-3mm}
\end{align}


In this paper we will carry out a similar analysis of the following related question: do there also exist exceptional times at which the path from the origin grows \textit{faster} than allowed by the law of iterated logarithm?  In this case we say $S_0^{\tau}$ is superdiffusive and call $\tau$ a superdiffusive exceptional time.  The question of the existence of such superdiffusive exceptional times has not been studied previously.  We will show that such exceptional times do in fact exist, and give a bound on how large such superdiffusive paths can get (see Theorems \ref{thm2} and \ref{thm3}).  We are also able to extend the subdiffusive results from \cite{OP}, showing the existence of exceptional times at which $S_0^{\tau}$ is subdiffusive in both directions, meaning $|S_0^{\tau}(t)| \leq j+K\sqrt{t}$ for all $t$.  


Now we will state our main results in the order in which they will appear.  The subdiffusive results will be presented first.  Sections \ref{struct} to \ref{thm1prf} are devoted to the proof of:
\begin{thm} \label{thm1} 
 For $K,j$ sufficiently large:
 \begin{align} \label{thm1eq}
\Prob\left(\exists \tau \in [0,1] \text{ s.t. } |S_0^{\tau}(t)| \leq j+K\sqrt{t} \text{   for all }t\geq 0\right)>0.
\end{align}
\end{thm}
\noindent An immediate consequence of this is:
\begin{cor} \label{cor1}
 For $K$ sufficiently large:
 \begin{align} \label{cor1eq}
\mathbb{P}\left(\exists \tau \in [0,1] \text{ s.t. } \limsup_{t \to \infty} \frac{|S_0^{\tau}(t)|}{\sqrt{t}} \leq K\right) = 1.
\end{align}
\end{cor}

Our study of superdiffusive exceptional times begins in Section \ref{st2}, where we prove:

\begin{thm} \label{thm2} 
 For $C>0$ sufficiently small:
 \begin{align} \label{thm2eq}
\Prob \left(\exists \tau \in [0,1] \text{ s.t. } \limsup_{t \to \infty} \frac{S_0^{\tau}(t)}{\sqrt{t \log(t)}} \geq C\right)=1.
\end{align}
\end{thm}

In Section \ref{2sidesup} we sketch a proof of the two-sided analogue of this theorem:  

\begin{thm} \label{thm2sup} 
 For $C>0$ sufficiently small:
 \begin{align} \label{thm2supeq}
\Prob \left(\exists \tau \in [0,1] \text{ s.t. } \limsup_{t \to \infty} \frac{S_0^{\tau}(t)}{\sqrt{t \log(t)}} \geq C \text{ and } \liminf_{t \to \infty} \frac{S_0^{\tau}(t)}{\sqrt{t \log(t)}} \leq -C \right)=1.
\end{align}
\end{thm}

The choice of $[0,1]$ for the interval of dynamical time is arbitrary.  The events in \eqref{cor1eq}, \eqref{thm2eq} and \eqref{thm2supeq} are (almost surely equal to) tail events with respect to the arrow processes.  This means that those sets of exceptional times will be a.s. empty or a.s. dense.  To see that Theorem \ref{thm1} still holds for any other choice of interval, first note that the process is stationary in $\tau$ so all that matters is the length of the interval.  If the probability in \eqref{thm1eq} is zero for a given choice of interval, clearly it must also be zero for any shorter interval.  However, any larger interval could be covered by multiple copies of the original interval, each of which would have probability zero of containing an exceptional time.  Thus the probability in \eqref{thm1eq} is zero for one choice of interval if and only if it is zero for all non-degenerate intervals.

Theorem \ref{thm2} is in some sense optimal, in that such exceptional times do not exist for large values of $C$.  In Section \ref{st3} we will prove:

\begin{thm} \label{thm3} 
 For $C>0$ sufficiently large:
 \begin{align} \label{thm3eq}
\Prob \left(\exists \tau \in [0,1] \text{ s.t. } \limsup_{t \to \infty} \frac{S_0^{\tau}(t)}{\sqrt{t \log(t)}} \geq C \right)=0.
\end{align}
\end{thm}


In the final section of the paper we study the Hausdorff dimensions of these various sets of exceptional times.  Section \ref{tshd} is devoted to two-sided subdiffusive exceptional times.  We look at the sets:
\begin{align}
 \{ \tau \in[0,\infty) &:\exists j \text{ s.t. } |S_0^{\tau}(t)|  \leq j+K\sqrt{t} \text{   for all $t\geq 0$} \}, \label{subset}\\
 \{ \tau \in[0,\infty) &:\limsup_{t \to \infty} |S_0^{\tau}(t)/\sqrt{t}| \leq K\label{sublimset} \},
\end{align}
\noindent and derive upper and lower bounds for their Hausdorff dimensions, as functions of $K$.  Our bounds are analogous to, and motivated by, those from \cite{OP} for the one-sided case.  As in the one-sided case, the dimensions tend to $1$ as $K$ goes to $\infty$.  In other words, the set of all two-sided subdiffusive exceptional times has Hausdorff dimension equal to one.  For small $K$ it is known that \eqref{subset} is empty, see Proposition 5.8 of \cite{OP}.  This implies \eqref{sublimset} is also empty for small $K$, see Section \ref{HD}.  Our analysis of \eqref{sublimset} is helped by noting that \eqref{sublimset} only depends on arrows with arbitrarily large time coordinate (almost surely).  This means \eqref{sublimset} can be analysed using tail events, allowing us to improve the lower bound slightly relative to the methods of \cite{OP}.  The two sets \eqref{subset} and \eqref{sublimset} have the same dimensions, except for at most countably many values of $K$ (see Section \ref{HD} for details).  In Section \ref{supdim} we look at the sets of superdiffusive exceptional times:
\begin{align}
 \{ \tau \in[0,\infty) &:\limsup_{t \to \infty} S_0^{\tau}(t)/\sqrt{t\log(t)} \geq C \}.
\end{align}
For these sets we are able to get a lower bound on Hausdorff dimension, but we do not have an upper bound at this time.  As a consequence of our lower bound we see that the dimension of the superdiffusive exceptional times tends to $1$ as $C$ goes to $0$, i.e. the set of all superdiffusive exceptional times has dimension one.
\end{section}

\begin{section}{Structure of the Proof of Theorem 1}\label{struct}$\quad$ 
As in \cite{OP}, we show that subdiffusivity occurs by showing that a series of ``rectangle events'' occur.  First, we define our rectangles.  Let $\gamma>1$ and $d_k=2(\lfloor\frac{\gamma^k}{2}\rfloor+1)$.  Let $R_0$ be the rectangle with vertices $(-d_0,0)$, $(+d_0,0)$, $(-d_0,d_0^2)$ and $(+d_0,d_0^2)$.  Given $R_k$ we take $R_{k+1}$ to be the rectangle of width $2d_{k+1}$ and height $d_{k+1}^2$, that is centered about the $t$-axis, and stacked on top of $R_k$ (see Figure \ref{fig: 2}).  An easy computation shows that the entire stack of rectangles lies between the graphs of $-j-K\sqrt{t}$ and $j+K\sqrt{t}$, where $j,K$ depend on $\gamma$.  For example, we can take $j=2,K=\gamma$, see Proposition \ref{Kbound} of Section 5.  Thus if $S_0^{\tau}$ stays within the stack, it will be subdiffusive in both directions.  

\begin{figure}[b!!!]
\begin{center}
 \includegraphics[scale=0.4]{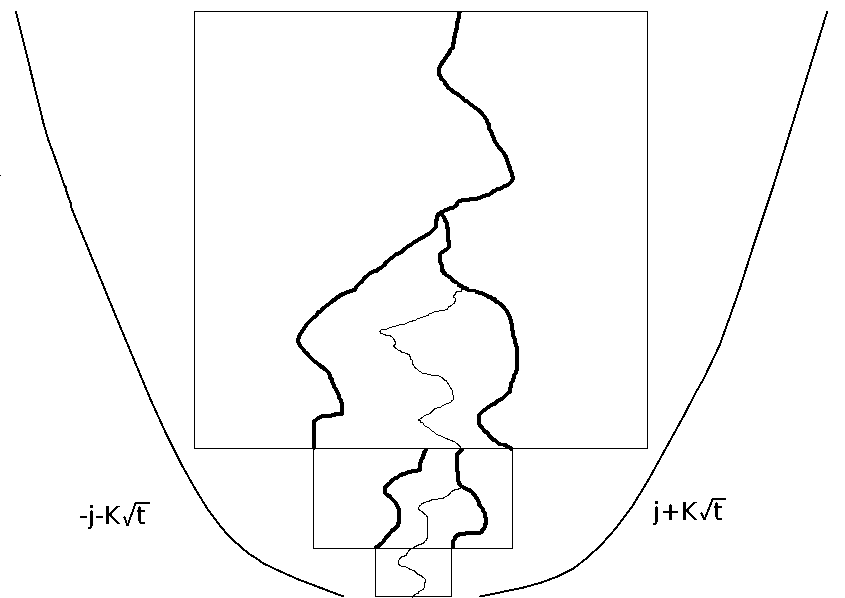}
\end{center}
 \caption{Rough sketch of the first three rectangles and paths for which the $B_k$'s occur.  The darker paths are the $S_{l_k}^{\tau}$'s and $S_{r_k}^{\tau}$'s.  The lighter path is $S_0^{\tau}$.}
 \label{fig: 2}
\end{figure}

\begin{rmk} \label{skew}
Notice that this gives a bound with left-right symmetry.  If we wish to study exceptional times where $-j_L-K_L\sqrt{t} \leq S_0^{\tau}(t) \leq j_R+K_R\sqrt{t}$, we can skew our rectangles.  This can be accomplished by horizontally scaling the left and right halves of each rectangle by $C_L$ and $C_R$, respectively (and rounding out to the nearest point in $\Z^2_{even}$).  For the sake of simplicity of our arguments (and notation) we will largely ignore the asymmetrical case.  However, it should be noted that our results easily extend to the asymmetrical case, using the above construction.
\end{rmk}

Let $t_k$ denote the time coordinate of the lower edge of $R_k$ (i.e. $t_k = d_0^2 + d_1^2 + ... + d_{k-1}^2$).  For $k\geq1$, let $l_k$ denote the upper left vertex of $R_{k-1}$ and $r_k$ the upper right vertex of $R_{k-1}$.  We would like to define our rectangle events, $B^{\tau}_k$, as:
\begin{align*}
B^{\tau}_0:=&\{|S_0^{\tau}(t)| \leq d_0 \; \; \forall t \in [0,t_1]\}, \\
B^{\tau}_k:=&\{|S_{l_k}^{\tau}(t)| \leq d_k \text{ and } |S_{r_k}^{\tau}(t)| \leq d_k \;\; \forall t \in [t_k,t_{k+1}] \} \; \; \text{for }k\geq1.
\end{align*}
Then on the event $\bigcap_{k\geq0} B^{\tau}_k$, $S_0^{\tau}$ will stay in the stack of rectangles, and thus be subdiffusive in both directions.  This follows from the discussion above, combined with the fact that paths in the discrete web do not cross. Thus if for some $\gamma$ we can show:
\begin{equation} \label{B}
 \mathbb{P}(\exists \tau \in [0,1] \text{ s.t. } \bigcap_{k\geq0} B^{\tau}_k(\gamma)\text{ occurs}) >0,
\end{equation}
then Theorem \ref{thm1} will follow immediately.

To prove \eqref{B}, we will need to understand the interaction between pairs of paths from the DyDW.  This can be described as a combination of coalescing (if the paths have the same dynamical time) and sticking (if the dynamical times differ).  Let $S^{\tau}_{z}$ be the path from $z=(x,t) \in \Z^2_{\text{even}}$ at dynamical time $\tau$, and let $S^{\tau'}_{z'}$ be the path from $z'=(x',t')$ at dynamical time $\tau'$.  The paths will evolve independently until they meet at some time $t^*\geq\text{Max}(t,t')$.  If $\tau=\tau'$, the paths coalesce when they meet, otherwise they ``stick''.  To be precise, let $x^*:=S^{\tau}_{z}(t^*)=S^{\tau'}_{z'}(t^*)$ and let $z^*=(x^*,t^*)(\in \Z^2_{\text{even}})$.  Then if the clock at $z^*$ has not rung in $(\tau,\tau']$ (WLOG assume $\tau<\tau'$), the two paths will follow the same arrow on $[t^*,t^*+1]$.  We will say the paths are sticking on $[t^*,t^*+1]$.  The paths continue to stick until they reach a site whose clock has rung, at which point they follow independent arrows.  Note that these independent arrows may agree, but this will not be considered sticking.

To prove Theorem \ref{thm1}, we would like to show \eqref{B}.  Unfortunately, we are not able to prove \eqref{B} directly.  The problem arises in the interaction between sticking and coalescing (to be specific, \eqref{Cd1}-\eqref{Cd3} fail for $B^{\tau}_k$, so we are unable to establish \eqref{prfeq2}).  To get around this, we construct a larger system where the relevant paths do not coalesce.  In addition to the main DyDW, $\mathcal{W}(\tau)$, we will need an independent, secondary DyDW, $\hat{\mathcal{W}}(\tau)$.  From now on, all ``arrows", ``clock rings", etc. should be understood to refer to $\mathcal{W}(\tau)$ (the main DyDW), unless otherwise specified.

Given $S_{l_k}^{\tau}$ and $S_{r_k}^{\tau}$ we want to construct non-coalescing versions, $X_{l_k}^{\tau}$ and $X_{r_k}^{\tau}$.  We accomplish this by letting $X_{l_k}^{\tau}=S_{l_k}^{\tau}$, and taking $X_{r_k}^{\tau}$ to be the path from $r_k$ that follows the arrows (from $\mathcal{W}(\tau)$) unless it meets $X_{l_k}^{\tau}$.  If $X_{r_k}^{\tau}$ meets $X_{l_k}^{\tau}$ at space-time $z^*=(x^*,t^*) \in \Z^2_{\text{even}}$, then on $[t^*,t^*+1]$ we let $X_{r_k}^{\tau}$ follow the arrow at $z^*$ from $\hat{\mathcal{W}}(\tau)$ (at dynamical time $\tau$).  At time $t^*+1$ we repeat this, following $\hat{\mathcal{W}}(\tau)$ if the paths are together, but following $\mathcal{W}(\tau)$ otherwise.  Continuing in this manner we get an independent pair of non-coalescing simple symmetric random walks $X_{l_k}^{\tau}$ and $X_{r_k}^{\tau}$.  Now we define new rectangle events, $C^{\tau}_k$:
\begin{align*} 
C^{\tau}_0:=&B^{\tau}_0, \\
C^{\tau}_k:=&\{|X_{l_k}^{\tau}(t)| \leq d_k \text{ and } |X_{r_k}^{\tau}(t)| \leq d_k \;\; \forall t \in [t_k,t_{k+1}] \} \; \; \text{for }k\geq1.
\end{align*}
Notice that $C^{\tau}_k$ implies $B^{\tau}_k$.  This is because the only difference between $X_{l_k}^{\tau}$,$X_{r_k}^{\tau}$ and $S_{l_k}^{\tau}$,$S_{r_k}^{\tau}$ is the (possible) extension of $X_{r_k}^{\tau}$ beyond the initial meeting point.  So if we can show:
\begin{equation} \label{C}
\mathbb{P}(\exists \tau \in [0,1] \text{ s.t. }\bigcap_{k\geq0} C^{\tau}_k\text{ occurs}) >0,
\end{equation}
then \eqref{B}, and thus Theorem \ref{thm1}, will follow immediately.  The next two sections will be devoted to proving \eqref{C}.

\end{section}

\begin{section}{A Decorrelation Bound}\label{decsec}$\quad$ 
Throughout this section we assume $\tau,\tau'\in[0,1],\tau<\tau'$ and we fix arbitrary $k\geq1, \gamma>1$.  We also translate the paths to start at $t=0$.  That is, we set $Y_{l}^{\tau}(t):=X_{l_k}^{\tau}(t_k+t)$ and $Y_{r}^{\tau}(t):=X_{r_k}^{\tau}(t_k+t)$ ($k$ is fixed so we drop it from the notation).  We will also consider diffusively rescaled versions of these paths, $\tilde{Y}_{l}^{\tau}(t):=Y_{l}^{\tau}(td_k^2)/d_k$ and $\tilde{Y}_{r}^{\tau}(t):=Y_{r}^{\tau}(td_k^2)/d_k$.  The relevant ``rectangle event'' is then:
\begin{align*}
 C^{\tau}:= &\{|Y_{l}^{\tau}(t)| \leq d_k \text{ and } |Y_{r}^{\tau}(t)| \leq d_k \;\; \forall t \in [0,d_k^2] \}  \\
=&\{|\tilde{Y}_{l}^{\tau}(t)| \leq 1 \text{ and } |\tilde{Y}_{r}^{\tau}(t)| \leq 1 \;\; \forall t \in [0,1] \}.
\end{align*}
See the beginning of Section \ref{struct} for definitions of $t_k,d_k,l_k,r_k$.  $X_{l_k}^{\tau},X_{r_k}^{\tau}$ are defined at the end of Section \ref{struct}.

Similarly to \cite{OP} we define $\Delta:=\frac{1}{d_k |\tau-\tau'|}$ (take their $\delta=d_k^{-1}$). As in \cite{OP}, the key ingredient for the proof of \eqref{C} is a decorrelation bound for the rectangle events:
\begin{prop}\label{dec}
 There exist $c,a \in (0,\infty)$ such that:
 \begin{center}
  $\Prob(C^{\tau} \cap C^{\tau'}) \leq \Prob(C^0)^2 + c\left(\Delta \right)^a \leq \Prob(C^0)^2 + c\left(\dfrac{1}{\gamma^k |\tau-\tau'|}\right)^a\!\!$, 
 \end{center}
 with $a,c$ independent of $ k, \tau$ and $\tau'$.
\end{prop}
\noindent Note that the second inequality is follows immediately from the definitions of $\Delta,d_k$.  The remainder of this section is devoted to proving the first inequality, and thus Proposition \ref{dec}.  The structure is similar to the proof of Proposition 3.1 from \cite{OP}, with a few necessary modifications.

As discussed in the previous section, paths from the DyDW at different dynamical times interact by sticking.  This sticking leads to dependence between the web paths.  Our modified paths (the $Y_{\tau}$'s) have their own version of sticking that is slightly more complicated.  To prove Proposition~\ref{dec} we will prove bounds for the amount of sticking, which will allow us to bound the dependence between the $C^{\tau}$'s.  We begin with some notation and definitions.

We call $n\in\Z$ a ``sticking time'' if a $Y^{\tau}$-path and a $Y^{\tau'}$-path follow the same arrow at time $n$.  For this to occur, a pair of paths from $Y_{l}^{\tau}, Y_{l}^{\tau'}, Y_{r}^{\tau}, Y_{r}^{\tau'}$ need to be at the same space-time location and follow the arrow from the same web ($\mathcal{W}$ or $\hat{\mathcal{W}}$, see the end of Section \ref{struct}).  In addition, this arrow must not have been updated in $[\tau,\tau']$.  This can happen in five ways:
\begin{enumerate}[(i)]
 \item $Y_{l}^{\tau}(n)=Y_{l}^{\tau'}(n)$ no ring in $[\tau,\tau']$,
 \item $Y_{l}^{\tau}(n)=Y_{r}^{\tau'}(n) \neq Y_{l}^{\tau'}(n)$ no ring in $[\tau,\tau']$,
 \item $Y_{l}^{\tau'}(n)=Y_{r}^{\tau}(n) \neq Y_{l}^{\tau}(n)$ no ring in $[\tau,\tau']$,
 \item $Y_{r}^{\tau}(n)=Y_{r}^{\tau'}(n) \neq Y_{l}^{\tau}(n),Y_{l}^{\tau'}(n)$ no ring in $[\tau,\tau']$,
 \item $Y_{r}^{\tau}(n)=Y_{r}^{\tau'}(n)=Y_{l}^{\tau}(n)=Y_{l}^{\tau'}(n)$ no $\hat{\mathcal{W}}$-ring in $[\tau,\tau']$.
\end{enumerate}
We will call (i) an $ll$(left-left)-sticking time, (ii) an $lr$-sticking time, (iii) an $rl$-sticking time, and (iv),(v) will both be $rr$-sticking times.  These names refer to which pair(s) of paths are sticking at time $n$.  

Given $s\in[0,\infty)$ let $n_s$ be the unique $n\in\Z$ such that $s\in[n,n+1)$.  We define:
\begin{align*}
  g(s) := \begin{cases}
   0 & \text{if $n_s$ is a sticking time} \\
   1 & \text{otherwise}
  \end{cases}
 \end{align*}
\begin{center}
 $G(t):=\int_0^t g(s) ds$.
\end{center}
We will also need:
\begin{align*}
  g_{ll}(s) := \begin{cases}
   0 & \text{if $n_s$ is an $ll$-sticking time} \\
   1 & \text{otherwise}
  \end{cases}
 \end{align*}
\begin{center}
 $G_{ll}(t):=\int_0^t g_{ll}(s) ds$,
\end{center}
and $G_{lr},G_{rl},G_{rr}$, which are defined analogously.

Notice that $t-G(t)$ is the amount of time spent sticking up to time $t$.  So if we make the time change $t \rightarrow t-G(t)$ we will include only the steps where a pair of paths from $Y_{l}^{\tau}, Y_{l}^{\tau'}, Y_{r}^{\tau}, Y_{r}^{\tau'}$ are sticking.  Similarly if we make the time change $t \rightarrow G(t)$ we will include only non-sticking steps.  This allows us to decompose the paths as:
\begin{align} \label{sep}
Y_{l}^{\tau}(t)=&Y_{l_d}^{\tau}(G(t))+Y_{l_s}^{\tau}(t-G(t)), \nonumber  \\
Y_{r}^{\tau}(t)=&Y_{r_d}^{\tau}(G(t))+Y_{r_s}^{\tau}(t-G(t)), \nonumber  \\
Y_{l}^{\tau'}(t)=&Y_{l_d}^{\tau'}(G(t))+Y_{l_s}^{\tau'}(t-G(t)), \nonumber \\
Y_{r}^{\tau'}(t)=&Y_{r_d}^{\tau'}(G(t))+Y_{r_s}^{\tau'}(t-G(t)), 
\end{align}
with $Y_{l_d}^{\tau}(0)=Y_l^{\tau}(0)=-d_{k-1}$, $Y_{r_d}^{\tau}(0)=Y_r^{\tau}(0)=d_{k-1}$, and $Y_{l_s}^{\tau}(0)=Y_{r_s}^{\tau}(0)=0$ (similarly for $\tau'$).  The $Y_{l_d}$'s and $Y_{r_d}$'s correspond to the time change $t \rightarrow G(t)$ and thus include only the non-sticking steps of each walk.  This means that $Y_{l_d}^{\tau},Y_{r_d}^{\tau}$ and $Y_{l_d}^{\tau'},Y_{r_d}^{\tau'}$ follow different, independent arrows, and thus are independent.  The $d,s$ subscripts of $l_d,l_s,r_d,r_s$ should not be confused with the $k$ of $l_k,r_k$ used earlier.  As mentioned above, $k$ is fixed throughout this section.  The $d,s$ subscripts merely denote the different components of the splitting defined in \eqref{sep}, ``$d$" for ``different" (i.e. non-sticking) and ``$s$" for ``sticking".

To make the above splitting work for the $\tilde{Y}$'s the appropriate rescaling of $G$ is $\bar{G}(t):=G(td_k^2)/d_k^2$.  We then make the  time changes $t \rightarrow t-\bar{G}(t)$  and $t \rightarrow \bar{G}(t)$.  We would like a bound for $t-\bar{G}(t)$, the total amount of sticking for the rescaled paths in $[0,t]$.  This is given by an adaptation of Lemma 3.4 from \cite{OP}.  Let $C(t)$ and $\bar{C}(t):=C(td_k^2)/d_k^2$ be defined as in \cite{OP}, i.e. such that $t-C(t)$ is the sticking time for $S_0^{\tau}$ and $S_0^{\tau'}$.  Lemma 3.4 from \cite{OP} gives:
\begin{align*}
\Prob \left( \sup_{t\in[0,1]} (t-\bar{C}(t)) \geq \Delta^{\beta} \right) \leq \tilde{c}\,\Delta^{1-\beta}
\end{align*}
for any $\beta \in (0,1)$, with $\tilde{c}$ independent of $k, \tau$ and $\tau'$.  This lemma follows from a bound on the expected number of sticking steps combined with the Markov inequality, see \cite{OP} for the details.  We will need a similar bound for $\bar{G}(t)$.
\begin{lem} \label{3.4}
For any $0<\beta<1$
\begin{center}$\Prob \left( \sup_{t\in[0,1]} (t-\bar{G}(t)) \geq \Delta^{\beta} \right) \leq c''\,\Delta^{1-\beta}$,
\end{center}
where $c''\in(0,\infty)$ is independent of $k, \tau$ and $\tau'$.
\end{lem}
\begin{proof}
Notice that by definition:
\begin{align} \label{Split} t-G(t) \leq \underset{(a)}{(t-G_{ll}(t))}+\underset{(b)}{(t-G_{lr}(t))}+\underset{(c)}{(t-G_{rl}(t))}+\underset{(d)}{(t-G_{rr}(t))} .
\end{align}
This is not an equality because the sticking times defined in (i)-(v) above are not mutually exclusive.  Our strategy will be to use Lemma 3.4 from \cite{OP} to bound the terms on the right side of \eqref{Split}.  We claim that each of $(a),(b),(c),(d)$ is stochastically bounded by $t-C(t)$ (given random variables $X,Y$, $X$ is said to stochastically bound $Y$ if $\Prob(Y>x) \leq \Prob(X>x)$ for all $x\in\R$).  For $(a)$ this is obvious, since $t-G_{ll}(t) \overset{d}{=} t-C(t)$ (equal in distribution).  This is because the $Y_l$'s are just translated web paths and the DyDW is invariant under space-time translations.  We now concentrate on $(d)$; $(b)$ and $(c)$ can be handled similarly.

We'd like to compare $t-G_{rr}(t)$, the amount of sticking for $Y_{r}^{\tau}$ and $Y_{r}^{\tau'}$, to $t-C(t)$, the amount of sticking for $S_0^{\tau}$ and $S_0^{\tau'}$.  We'll accomplish this by constructing coupled versions of the two processes.  In both cases there are two paths that alternate between identical sticking sections and independent non-sticking sections.  To be specific, we take $T_0=T^*_0:=0$ and for $k\geq0$ define:
\begin{align*}
T_{2k+1}:=& \inf \{k \geq T_{2k} : \text{ The clock at $S^{\tau}_0(k)=S^{\tau'}_0(k)$ rings in $[\tau,\tau']$} \}, \\
T_{2k+2}:=&\inf \{ k > T_{2k+1} : \; S_0^{\tau}(k)=S_0^{\tau'}(k) \}, \\
\Delta_k:=&T_{2k+1}-T_{2k} \geq 0,\; \Gamma_k:=T_{2k+2}-T_{2k+1} \geq 1,
\end{align*} \\
and:
\begin{align*}
T^*_{2k+1}:=&\inf \{k \geq T^*_{2k} : \text{ $k$ is not an $rr$-sticking time} \}, \\
T^*_{2k+2}:=&\inf \{k > T^*_{2k+1} : \; Y_{r}^{\tau}(k)=Y_{r}^{\tau'}(k) \}, \\
\Delta^*_k:=&T^*_{2k+1}-T^*_{2k} \geq 0,\; \Gamma^*_k:=T^*_{2k+2}-T^*_{2k+1} \geq 1.
\end{align*}
Then on $[T^{(*)}_{2k},T^{(*)}_{2k+1}]$ we have $S_0^{\tau}$ and $S_0^{\tau'}(Y_{r}^{\tau}$ and $Y_{r}^{\tau'})$ sticking for $\Delta^{(*)}_k$ steps, while on $[T^{(*)}_{2k+1},T^{(*)}_{2k+2}]$ they move independently until meeting at $T^{(*)}_{2k+2}$.  Parentheses are used here to indicate two cases (with a $*$ and without a $*$) simultaneously.  We will continue to use this convention, hopefully the meaning will be clear to the reader.  Next notice that $\Gamma_k$ and $\Gamma_k^*$ have the same distribution, they are both excursion times for pairs of independent random walks.  So we may take $\Gamma_k= \Gamma_k^*$ for our coupled versions.  To compare $\Delta_k,\Delta_k^*$, notice that:
\begin{align*}
\Prob \left( \Delta_k^{(*)} \geq j \right) =& \prod_{i=1}^j\Prob \left( \Delta_k^{(*)} \geq i | \Delta_k^{(*)} \geq i-1 \right) 
\end{align*}
and:
\vspace{-3mm}
\begin{align} \label{del1}
\Prob \left( \Delta_k^{*} \geq i | \Delta_k^{*} \geq i-1 \right) \leq \Prob \left( \Delta_k \geq i | \Delta_k \geq i-1 \right) \text{   for all $i\geq 1$},
\end{align}
so:
\vspace{-3mm}
\begin{align} \label{del2}
\Prob \left( \Delta_k^{*} \geq j \right) \leq \Prob \left( \Delta_k \geq j \right) \text{    for all } j,k \geq 0.
\end{align} 
To see \eqref{del1}, consider that $\Prob \left( \Delta_k \geq i | \Delta_k \geq i-1 \right)$ is just the probability of no clock ring in $[\tau, \tau']$.  For $\Delta_k^*$, we have the probability that $Y_{r}^{\tau}=Y_{r}^{\tau'} \neq Y_{l}^{\tau},Y_{l}^{\tau'}$ and there is no $\mathcal{W}$-ring, or $Y_{r}^{\tau}=Y_{r}^{\tau'}=Y_{l}^{\tau}=Y_{l}^{\tau'}$ and there is no $\hat{\mathcal{W}}$-ring.  These are disjoint events and the clocks are independent of the positions of previous arrows, so this is bounded by the probability of no clock ring.

Combining this with the above observations, we can couple $\Delta_k,\Delta_k^*$ and $\Gamma_k,\Gamma_k^*$ such that $\Delta_k^* \leq \Delta_k$ and $\Gamma_k=\Gamma_k^*$.  This means that the $rr$-sticking sections are shorter than the $S_0^{\tau},S_0^{\tau'}$ sticking sections, while the independent sections have the same length.  This implies $t-G_{rr}(t) \leq t-C(t)$ for the coupled versions, which shows $(d)$ is stochastically bounded by $t-C(t)$.  This can be proven for $(b),(c)$ by a nearly identical coupling argument, where the portion of the left/right paths after their first meeting is coupled with $S_0^{\tau},S_0^{\tau'}$.  So we've shown that $(a),(b),(c),(d)$ are each stochastically bounded by $t-C(t)$.  Combining this with \eqref{Split} we get:

\begin{align*}
\Prob \left( \sup_{t\in[0,1]} (t-\bar{G}(t)) \geq \Delta^{\beta} \right) =&
\Prob \left( \sup_{t\in[0,d_k^2]} (t-G(t)) \geq d_k^2\Delta^{\beta} \right) \\
\leq & 4 \Prob \left( \sup_{t\in[0,d_k^2]} (t-C(t)) \geq d_k^2\dfrac{\Delta^{\beta}}{4} \right) \text{   (using \eqref{Split} and above paragraph)}\\
=&4 \Prob \left( \sup_{t\in[0,1]} (t-\bar{C}(t)) \geq \dfrac{\Delta^{\beta}}{4} \right) \\
\leq & 4 \tilde{c} \, \left( \dfrac{\Delta}{4^{1/\beta}} \right)^{1-\beta} \text{    (by Lemma 3.4 from \cite{OP})}  \\
=& c''\, \Delta^{1-\beta}.
\end{align*}
This completes the proof since $\tilde{c}$, and thus $c''$, is independent of $k, \tau$ and $\tau'$.
\end{proof} 

Now we define $C^{\tau}_d$ to be the rectangle event for $Y_{l_d}^{\tau},Y_{r_d}^{\tau}$.  That is:
\begin{align*}
 C^{\tau}_d:= &\{|Y_{l_d}^{\tau}(t)| \leq d_k \text{ and } |Y_{r_d}^{\tau}(t)| \leq d_k \;\; \forall t \in [0,d_k^2] \}  \\
=&\{|\tilde{Y}_{l_d}^{\tau}(t)| \leq 1 \text{ and } |\tilde{Y}_{r_d}^{\tau}(t)| \leq 1 \;\; \forall t \in [0,1] \}.
\end{align*}
Given $r>0$ we define the $r$-approximations of our rectangle events as:
\begin{align*}
\{C^{\tau}_{(d)}+r\}:=&\{|Y_{l_{(d)}}^{\tau}(t)| \leq (1+r)d_k \text{ and } |Y_{r_{(d)}}^{\tau}(t)| \leq (1+r)d_k \;\; \forall t \in [0,d_k^2] \}  \\
=& \{|\tilde{Y}_{l_{(d)}}^{\tau}(t)| \leq 1+r \text{ and } |\tilde{Y}_{r_{(d)}}^{\tau}(t)| \leq 1+r \;\; \forall t \in [0,1] \}.
\end{align*}
Recall that $Y_{l_d}^{\tau},Y_{r_d}^{\tau}$ are independent of $Y_{l_d}^{\tau'},Y_{r_d}^{\tau'}$, and therefore:
\begin{align} \label{Cd1}
C^{\tau}_d(\{C^{\tau}_{d}+r\}) \text{ is independent of } C^{\tau'}_d(\{C^{\tau'}_{d}+r\}).
\end{align}
We also have:
\begin{align} \label{Cd2}
 (Y_{l_d}^{\tau},Y_{r_d}^{\tau}) \stackrel{d}{=} (Y_{l}^{\tau},Y_{r}^{\tau}),
\end{align}
since both are just pairs of independent random walks.  So: 
\begin{align} \label{Cd3}
\Prob(C^{\tau}_d)=\Prob(C^{\tau})=\Prob(C^{0}).
\end{align}

We will need the following adaptation of Lemma 3.3 from \cite{OP}:
\begin{lem} \label{3}
Given any $\alpha<1/2$, there is $c'\in(0,\infty)$ independent of $\Delta, k$ such that: 
\begin{center}
$\Prob\left(\{C^{\tau}_{d}+\Delta^{\alpha}\} \setminus  C^{\tau}_{d}\right)\leq c'\Delta^{\alpha}$.
\end{center}
\end{lem}
\begin{proof}
\begin{align*}
\Prob\left(\{C^{\tau}_{d}+\Delta^{\alpha}\} \setminus  C^{\tau}_{d}\right) \leq \; &\Prob\left( \inf_{t\in[0,1]} \tilde{Y}_{l_{d}}^{\tau}(t) \in [-1-\Delta^{\alpha},-1)\right)+
\Prob\left( \sup_{t\in[0,1]} \tilde{Y}_{l_{d}}^{\tau}(t) \in (1,1+\Delta^{\alpha}] \right)\\
+&\Prob\left( \inf_{t\in[0,1]} \tilde{Y}_{r_{d}}^{\tau}(t) \in [-1-\Delta^{\alpha},-1) \right) 
+\Prob\left( \sup_{t\in[0,1]} \tilde{Y}_{r_{d}}^{\tau}(t) \in (1,1+\Delta^{\alpha}] \right).
\end{align*}
Now each of the four terms on the right is bounded by $c \, \Delta^{\alpha}$.  This follows exactly as in the proof of Lemma 3.3 in \cite{OP}.  To see this, note that the $\tilde{Y}$'s are simple symmetric random walks started at $\pm d_{k-1}/d_k \in [-1,1]$, diffusively rescaled by $\delta=d_k^{-1}$.  We can thus approximate the $\tilde{Y}$'s by Brownian motion paths (for details see \cite{OP13} and \cite{OP}, Lemma 3.3).  The result then follows, as the maximum (minimum) process of a Brownian motion has a bounded probability density function.
\end{proof}

The final ingredient for the proof of Proposition \ref{dec} is a bound on the modulus of continuity of a random walk.  This is given by Lemma 3.5 from \cite{OP}:
\begin{lem}(Lemma 3.5, \cite{OP}) \label{3.5} \\
Let $S(t)$ be a simple symmetric random walk and define $\tilde{S}(t):=S(t/\delta^2)\delta$. \\
Let $\omega_{\tilde{S}}(\epsilon):= \sup_{s,t\in[0,1], |s-t|< \epsilon}|\tilde{S}(t)-\tilde{S}(s)|$ be the modulus of continuity of $\tilde{S}$.  Let $\alpha,\beta\in(0,\infty)$ be such that $\beta/2>\alpha$.  For any $r\geq 0$, there exists $c$ (independent of $\Delta$ and $\delta$) such that:
\begin{center}
$ \Prob \left( \omega_{\tilde{S}}(\Delta^{\beta}) \geq \dfrac{\Delta^{\alpha}}{2} \right) \leq c \, \Delta^r  $.
\end{center}
\end{lem}
\noindent This is a consequence of the Garsia-Rodemich-Rumsey inequality \cite{GRR}.  For a proof see \cite{OP}.

We may now prove Proposition 1.  The remaining steps are nearly identical to the proof of Proposition 3.1 from \cite{OP} (see the end of Section 3).  We include them for the sake of completeness.

For any $0<\alpha<1/2$, we have:
\begin{align} \label{prfeq1}
 \Prob \left( C^{\tau} \cap C^{\tau'} \right) \leq & \Prob \left( \{ C_d^{\tau} + \Delta^{\alpha} \} \cap \{ C_d^{\tau'} + \Delta^{\alpha} \} \right) \nonumber \\
+&2\Prob \left(  C^{\tau}  \setminus \{ C_d^{\tau} + \Delta^{\alpha} \} \right), 
\end{align}
where we used the equidistribution of $(C^{\tau},\{ C_d^{\tau} + \Delta^{\alpha} \})$ and $(C^{\tau'},\{ C_d^{\tau'} + \Delta^{\alpha} \})$.  Using \eqref{Cd1}-\eqref{Cd3} we get:
\begin{align} \label{prfeq2}
\Prob ( \{ C_d^{\tau} + \Delta^{\alpha} \} \cap \{ C_d^{\tau'} + \Delta^{\alpha} \} ) =& \Prob ( \{ C_d^{\tau} + \Delta^{\alpha} \}) \Prob( \{ C_d^{\tau'} + \Delta^{\alpha} \} )\nonumber \\
\leq & \Prob (  C_d^{\tau})^2 + 2\Prob(  \{C^{\tau}_{d}+\Delta^{\alpha}\} \setminus  C^{\tau}_{d} ) \nonumber \\
=& \Prob (  C^0)^2 + 2\Prob(  \{C^{\tau}_{d}+\Delta^{\alpha}\} \setminus  C^{\tau}_{d} ).
\end{align}
Combined with Lemma \ref{3} this gives:
\begin{equation} \label{prfeq3}
\Prob ( \{ C_d^{\tau} + \Delta^{\alpha} \} \cap \{ C_d^{\tau'} + \Delta^{\alpha} \} ) \leq \Prob (  C^0)^2 +2c'\Delta^{\alpha}.
\end{equation}
Now that we have \eqref{prfeq1} and \eqref{prfeq3} we just need $\hat{c},a'$ such that:
\begin{equation}\label{prfeq4}
\Prob \left(  C^{\tau}  \setminus \{ C_d^{\tau} + \Delta^{\alpha} \} \right) \leq \hat{c} \, \Delta^{a'}.
\end{equation}
Recall the splitting of the $Y^{\tau}$'s given by \eqref{sep}.  Analogous considerations for the $\tilde{Y}^{\tau}$'s gives:
\begin{align}
\tilde{Y}_{l}^{\tau}(t)=&\tilde{Y}_{l_d}^{\tau}(\bar{G}(t))+\tilde{Y}_{l_s}^{\tau}(t-\bar{G}(t)) \nonumber \\
=&\tilde{Y}_{l_d}^{\tau}(t) + [\tilde{Y}_{l_d}^{\tau}(\bar{G}(t)) - \tilde{Y}_{l_d}^{\tau}(t)] +\tilde{Y}_{l_s}^{\tau}(t-\bar{G}(t)), \label{peq5} \\
\tilde{Y}_{r}^{\tau}(t)=&\tilde{Y}_{r_d}^{\tau}(t) + [\tilde{Y}_{r_d}^{\tau}(\bar{G}(t)) - \tilde{Y}_{r_d}^{\tau}(t)] +\tilde{Y}_{r_s}^{\tau}(t-\bar{G}(t)). \label{peq6}
\end{align}
Notice that all the $\tilde{Y}$'s appearing in \eqref{peq5}, \eqref{peq6} are simple symmetric random walks rescaled by $\delta=d_k^{-1}$, as in Lemma \ref{3.5}.  Also, we've taken $\alpha<1/2$, so we may choose $0<\beta<1$ such that $\beta/2>\alpha$.  Then:
\begin{align}
\Prob \left(  C^{\tau}  \setminus \{ C_d^{\tau} + \Delta^{\alpha} \} \right) \leq& \, 
\Prob \left( |\tilde{Y}_l^{\tau}-\tilde{Y}_{l_d}^{\tau}|_{\infty} \geq \Delta^{\alpha} \right) +
\Prob \left( |\tilde{Y}_r^{\tau}-\tilde{Y}_{r_d}^{\tau}|_{\infty} \geq \Delta^{\alpha} \right) \nonumber \\
\leq& \, \Prob \left( |\tilde{Y}_{l_d}^{\tau}(\bar{G}(t))-\tilde{Y}_{l_d}^{\tau}(t)|_{\infty} \geq \dfrac{\Delta^{\alpha}}{2} \right) + 
\Prob \left( |\tilde{Y}_{l_s}^{\tau}(t-\bar{G}(t))|_{\infty} \geq \dfrac{\Delta^{\alpha}}{2} \right) \nonumber \\
&+  \Prob \left( |\tilde{Y}_{r_d}^{\tau}(\bar{G}(t))-\tilde{Y}_{r_d}^{\tau}(t)|_{\infty} \geq \dfrac{\Delta^{\alpha}}{2} \right) + 
\Prob \left( |\tilde{Y}_{r_s}^{\tau}(t-\bar{G}(t))|_{\infty} \geq \dfrac{\Delta^{\alpha}}{2} \right) \nonumber \\
\leq& \, 4 \Prob \left( \omega_{\tilde{S}}(\Delta^{\beta}) \geq \dfrac{\Delta^{\alpha}}{2} \right) 
+ 4\Prob \left( \sup_{t\in[0,1]} (t-\bar{G}(t)) \geq \Delta^{\beta} \right) \nonumber \\
\leq& \, 4c \, \Delta^r + 4c''\,\Delta^{1-\beta}, \label{peqend}
\end{align}
where $|\cdot|_{\infty}$ denotes the sup norm restricted to $[0,1]$.  The last inequality follows from Lemmas \ref{3.4} and \ref{3.5}.  This completes the proof of Proposition \ref{dec}.
\end{section}

\begin{section}{Proof of Theorem \ref{thm1}}\label{thm1prf}
\quad Now that we have Proposition \ref{dec} we are almost ready to prove Theorem \ref{thm1}.  We'd like to show the existence of exceptional times at which $\bigcap_{k\geq 0}C^{\tau}_k$ occurs (see the end of Section \ref{struct}).  We just need one more Lemma from \cite{OP}:
\begin{lem} \label{4.3} (Lemma 4.3, \cite{OP}) There exists $c\in(0,\infty)$ such that for $\tau,\tau'\in[0,1]$, $\forall n \geq 0$:
\begin{align*}
\prod_{k=0}^n  \dfrac{\Prob(C^{\tau}_k \cap C^{\tau'}_k)}{\Prob(C_k)^2}  \leq c\dfrac{1}{|\tau-\tau'|^b},
\end{align*}
where $C_k:=C_k^0$ and $b=\log(\sup_k[\Prob(C_k)^{-1}])/\log \gamma >0$.
\end{lem}  
This was established in \cite{OP} for a different collection of rectangle events, $A_k$.  The key idea is to split the product at $N_0 := \lfloor -\log(|\tau - \tau'|)/\log(\gamma)\rfloor + 1$.  To make their proof work for $C_k$, we just need $a,c$ such that:
\begin{align}\label{4.3.1}
 \Prob(C^{\tau}_k \cap C^{\tau'}_k) \leq \Prob(C_k)^2 +c\left(\dfrac{1}{\gamma^k |\tau-\tau'|}\right)^a  \;\; \forall \tau,\tau' \in [0,1], k \geq 0,
\end{align}
and:
 \begin{align} \label{4.3.2}
 \sup_k[\Prob(C_k)^{-1}])< \infty.
\end{align}
\eqref{4.3.1} follows from Proposition \ref{dec}.  To see \eqref{4.3.2}, notice that the rectangles $R_k$ scale diffusively, and therefore $\Prob (C_k) \longrightarrow \Prob(C_{\infty})$, the probability of the corresponding rectangle event for Brownian motion paths.  Lemma \ref{4.3} now follows exactly as in \cite{OP}; we'll sketch the main steps here.  Using \eqref{4.3.1} and $\gamma^{N_0}|\tau - \tau'| \geq 1$ (see the definition of $N_0$ above) we can bound:
\begin{align*}
 \prod_{k=N_0+1}^{\infty}  \dfrac{\Prob(C^{\tau}_k \cap C^{\tau'}_k)}{\Prob(C_k)^2}
\end{align*}
by a constant which does not depend on $\tau,\tau'$.  Using the definition of $N_0$ and \eqref{4.3.2} we get:
\begin{align*}
 \prod_{k=0}^{N_0}  \dfrac{\Prob(C^{\tau}_k \cap C^{\tau'}_k)}{\Prob(C_k)^2} \leq c'/|\tau - \tau'|^b
\end{align*}
See \cite{OP} for further details.  A similar argument is given in the proof of Proposition \ref{supdimprop} below. 

Theorem \ref{thm1} now follows as in \cite{OP},\cite{OP3}.  We will repeat their arguments for the sake of completeness.  The Cauchy-Schwartz inequality and Lemma \ref{4.3} give, $\forall n\geq 0$:
\begin{align}
\Prob \left( \int_0^1 \prod_{k=0}^n \1_{C^{\tau}_k} d\tau >0 \right) 
\geq& \dfrac{\left(\E \left[\int_0^1 \prod_{k=0}^n \1_{C^{\tau}_k} d\tau \right] \right)^2 }
{\E \left[ \left( \int_0^1 \prod_{k=0}^n \1_{C^{\tau}_k} d\tau \right)^2 \right] }  \label{thmp1} \\
=& \left[ \int_0^1 \int_0^1 \prod_{k=0}^n  \dfrac{\Prob(C^{\tau}_k \cap C^{\tau'}_k)}{\Prob(C_k)^2} d\tau d\tau' \right]^{-1} \label{thmp2} \\
\geq& c^{-1} \left[  \int_0^1 \int_0^1 \dfrac{1}{|\tau-\tau'|^b} d\tau d\tau' \right]^{-1} , \label{thmp3} 
\end{align}
where \eqref{thmp2} comes from the independence of the arrow configurations in different $R_k$'s and the stationarity of $\tau \longrightarrow \mathcal{W}(\tau)$.  We would like to show that \eqref{thmp3} is strictly positive.  Lemma \ref{4.3} gave:
\begin{align*}
b=&\log(\sup_k[\Prob(C_k)^{-1}])/\log \gamma. 
\end{align*}
Recall that the $R_k$'s, and thus the $\Prob(C_k)$'s, depend on $\gamma$ (see the beginning of Section \ref{struct} for the definition of $R_k$).  As $\gamma$ increases, $R_0$ remains the same, while for $k\geq 1$, $R_k$ scales diffusively.  The size of $R_{k-1}$ relative to $R_{k}$ also tends to zero, so the starting points of $X_{l_k}^{\tau},X_{r_k}^{\tau}$ converge to the center of the rectangle when diffusively rescaled.  This implies that as $\gamma$ goes to infinity, $\sup_k[\Prob(C_k)^{-1}]$ converges to $\max \{\Prob(C_0)^{-1}, \Prob(C^*)^{-1}\}$, where $C^*$ is the rectangle event for two independent Brownian motions started in the center.  So for $\gamma$ sufficiently large we have $b<1$, and thus $|\tau-\tau'|^{-b}$ integrable on $[0,1]\times[0,1]$.  \eqref{thmp1}-\eqref{thmp3} then imply:
\begin{align}\label{thmp4}
\inf_n \Prob \left( \int_0^1 \prod_{k=0}^n \1_{C^{\tau}_k} d\tau >0 \right) \geq p >0.
\end{align}
Letting $E_n:=\{ \tau \in [0,1] : \; \bigcap_{k=0}^n C_k^{\tau} \text{ occurs}\}$, \eqref{thmp4} then implies $\Prob ( \bigcap_{n=0}^{\infty} \{E_n \neq \emptyset \} ) \geq p>0$.  Notice that the $E_n$ are decreasing in $n$.  So if the $E_n$ were closed, this would imply $\Prob ( (\bigcap_{n=0}^{\infty} E_n) \neq \emptyset ) \geq p>0$ and \eqref{C}, and thus Theorem \ref{thm1}, would follow.

Unfortunately, the $E_n$ are not closed.  This is handled as in \cite{OP},\cite{OP2}(Lemma 3.2) by noting that the $E_n$ are nested collections of intervals, and their endpoints must correspond to clock rings at some site in $\mathcal{W}$ or $\hat{\mathcal{W}}$.  The dynamics are stationary and there are only countably many clock rings. Furthermore, the locations of the clock rings at a given site are independent of all other arrows' directions, and when a site's clock rings its arrow is re-sampled according to an independent, $p=1/2$ Bernoulli random variable.  So conditioned on $\tau$ being a clock ring both $\mathcal{W}(\tau)$ and $\hat{\mathcal{W}}(\tau)$ are distributed as DW's.  An important consequence of this discussion is that if an event has probability zero of occurring in a pair of DW's, then there is probability zero that this event will occur at any clock ring for the pair of DyDW's ($\mathcal{W}$, $\hat{\mathcal{W}}$). This will imply that, almost surely:
\begin{align} \label{thmp5}
\bigcap_{n=0}^{\infty} E_n = \bigcap_{n=0}^{\infty} \bar{E}_n .
\end{align} 
To see this, consider the event that \eqref{thmp5} fails, i.e., the event that there exists a $\tau^*$ that is in the right set but not the left set.  Then for some $m \geq 0$, $\tau^*$ is in $\bar{E}_m$, but not $E_m$.  This implies that $\tau^*$ is the right endpoint to an interval from $E_m$ (by right continuity), and thus must be a switching time for exactly one arrow, $\xi^{\tau}_*$, from $\mathcal{W}$ or $\hat{\mathcal{W}}$.  Now, since the $E_n$ are nested and $\tau^*$ is in $\bar{E}_n$ for all $n$, $\tau^*$ must also be a right endpoint to an interval from $E_n$ (but not in $E_n$) for all $n\geq m$.  This means that for all $n \geq m$ there is an $\epsilon>0$ (which depends on $n$) such that $\bigcap_{k=0}^n C_k^{\tau}$ occurs for $\tau \in [\tau^*-\epsilon, \tau^*)$, but not at $\tau = \tau^*$.  However, $\bigcap_{k=0}^n C_k^{\tau}$ ceases to occur at $\tau = \tau^*$ only due to the resetting of $\xi^{\tau}_*$ (almost surely, since with probability one no two arrows will switch at the same time).  This means that by switching the value of $\xi^{\tau^*}_*$ we will cause $\bigcap_{k=0}^{n} C_k^{\tau^*}$ to occur for all $n$, i.e. $\bigcap_{k=0}^{\infty} C_k^{\tau^*}$ will occur.  In other words, $\xi^{\tau^*}_*$ is pivotal for $\bigcap_{k=0}^{\infty} C_k^{\tau^*}$, where ``pivotal" means that the event occurs for one value of $\xi^{\tau^*}_*$ but not the other.  However, the event that an arrow is pivotal for $\bigcap_{k=0}^{\infty} C_k$ has probability zero to occur in a pair of DW's, since $\bigcap_{k=0}^{\infty} C_k$ itself has probability zero and the probability that an arrow is pivotal for an event is bounded by twice the probability of the event.  Also, $\tau^*$ is a switching time and thus also a clock ring.  So based on the discussion above, there is probability zero that $\xi^{\tau^*}_*$ (or any arrow) is pivotal for $\bigcap_{k=0}^{\infty} C_k^{\tau*}$.  We've shown that the existence of a $\tau^*$ that is in the right set of \eqref{thmp5} but not the left set implies the occurrence of a probability zero event, i.e. the event $\{\eqref{thmp5} \text{ fails}\}$ has probability zero.  This proves that \eqref{thmp5} holds almost surely and Theorem \ref{thm1} then follows from the discussion in the previous paragraph.
\end{section}

\begin{section} {Proof of Theorems \ref{thm2} and \ref{thm2sup}}\label{st2}$\quad$ 
In this section we prove Theorem \ref{thm2}, which states that for $C>0$ sufficiently small we have:
\begin{align*}
\Prob \left(\exists \tau \in [0,1] \text{ s.t. } \limsup_{t \to \infty} \frac{S_0^{\tau}(t)}{\sqrt{t \log(t)}} \geq C\right)=1.
\end{align*}
We begin by defining new rectangle events, $\hat{A}_k^{\tau}$.  The new rectangles, $\hat{R}_k$, will be similar to the $R_k$ defined in Section \ref{struct} (see Figure \ref{fig: 2}), but they will be wider and grow faster in $k$.  We will let $\gamma>1$ and define:
\begin{align*}
\hat{d}_k=2(\lfloor\frac{\gamma^{\gamma^k}}{2}\rfloor+1).
\end{align*}
It is possible to prove Theorem \ref{thm2} using the same $d_k$ from Section \ref{struct}.  However, the faster growth of the $\hat{d}_k$ is necessary for the results of Sections \ref{2sidesup} and \ref{supdim}.

Now let $C>0$ and introduce $\hat{w}_k=2(\lfloor C\sqrt{\log(\hat{d}_k^2)\hat{d}_k^2}/2\rfloor+1)$.  Take $\hat{R}_0$ to be the rectangle with vertices $(\hat{w}_0,0)$, $(-\hat{w}_0,0)$, $(\hat{w}_0,\hat{d}_0^2)$ and $(-\hat{w}_0,\hat{d}_0^2)$.  $\hat{R}_{k+1}$ will be the rectangle of width $2\hat{w}_{k+1}$ and height $\hat{d}_{k+1}^2$, stacked on top of $\hat{R}_k$ and centered about the $t$-axis.  Let $\hat{l}_k,\hat{r}_k$ be the upper left and upper right corners of $\hat{R}_{k-1}$.  $\hat{t}_k$ will be the time coordinate of the lower edge of $\hat{R}_k$.  Then for $k\geq1$ we define:
\begin{align*}
\hat{A}_k^{\tau}:=\{S_{\hat{l}_k}^{\tau}(\hat{t}_{k+1}) > \hat{w}_k\}.
\end{align*}
Notice that on $\hat{A}_k^{\tau}$ we must have either:
\begin{align*}
S_0^{\tau}(\hat{t}_k)<-\hat{w}_{k-1}
 \text{ or }
S_0^{\tau}(\hat{t}_{k+1})>\hat{w}_k.
\end{align*}
Now $\hat{w}_k\geq C_k' \sqrt{\log(\hat{t}_{k+1})\hat{t}_{k+1}}$ with $C_k'<C$.   Furthermore, $\hat{t}_{k+1} / \hat{d}_k^2 \to 1$ when $k \to \infty$, so we may choose $C_k'$ such that $C_k' \to C$.  So for a given $\tau$, if $\hat{A}_k^{\tau}$ occurs for infinitely many $k$ this will imply:
\begin{align}
\limsup_{t \to \infty} \frac{S_0^{\tau}(t)}{\sqrt{t \log(t)}} \geq C
 \text{ or }
\liminf_{t \to \infty} \frac{S_0^{\tau}(t)}{\sqrt{t \log(t)}} \leq -C.
\end{align}
Using symmetry we could then say that both types of exceptional times must in fact exist.  So our strategy for proving the existence of superdiffusive exceptional times will be to show that for $C$ small, we have:
\begin{equation} \label{Ahat}
\Prob \left(\exists \tau \in [0,1] \text{ s.t. } \hat{A}_k^{\tau}\text{ occurs infinitely often}\right) =1.
\end{equation}

To begin, we define $\hat{E}_k:= \{\tau \text{ s.t. }\hat{A}_k^{\tau}\}$ and examine $\Prob (\hat{E}_k \cap [a,b]\neq \emptyset)$.

\begin{prop} \label{cl1}
For $C$ sufficiently small,
\begin{align}
\Prob (\hat{E}_k \cap [a,b] \neq \emptyset) \to 1
\end{align}
as $k \to \infty$ for all $a < b$ in $[0,\infty)$.
\end{prop}

This Proposition implies Theorem \ref{thm2}.  Notice that the $\hat{E}_k$ are independent, and each $\hat{E}_k$ is a disjoint union of half open intervals.  Let $[a_0,b_0]=[0,1]$.  Proposition \ref{cl1} implies $\Prob (\hat{E}_k \cap [a_0,b_0] \neq \emptyset)  \to 1$ as $k \to \infty$.  So, almost surely we will have some (random) $k_1$ such that $\hat{E}_{k_1} \cap [a_0,b_0] \neq \emptyset$ and we can choose $[a_1,b_1] \subset \hat{E}_{k_1} \cap [a_0,b_0]$.  By independence we then have $\Prob (\hat{E}_k \cap [a_1,b_1] \neq \emptyset) \to 1$ as $k \to \infty$.  Then almost surely for some $k_2>k_1$, we'll have $[a_2,b_2] \subset \hat{E}_{k_2} \cap [a_1,b_1]$.  Continuing in this manner we get a random nested sequence of non-empty closed intervals, $\{[a_k,b_k]\}_{k\geq 0}$.  Furthermore,  $\bigcap_{k\geq 0}[a_k,b_k] \neq \emptyset$ almost surely, and for $\tau \in \bigcap_{k\geq 0}[a_k,b_k]$ we know that $\hat{A}_k^{\tau}$ occurs for infinitely many $k$.  This proves \eqref{Ahat}, and thus Theorem \ref{thm2}.

Now let $\hat{\Delta}:=\dfrac{1}{\hat{d}_k |\tau-\tau'|}$.  To prove Proposition \ref{cl1} we need the following decorrelation bound, which is the natural analog of Proposition 3.1 from \cite{OP}, and Proposition \ref{dec} above.  The decorrelation bound is essentially the same as in the subdiffusive case, but we will need to use it in a different way to account for the fact that $\Prob(\hat{A}_k^{\tau}) \to 0$.
\begin{lem}\label{hdec}
 There exist $c',a'\in (0,\infty)$ such that:
 \begin{center}
$\Prob(\hat{A}_k^{\tau} \cap \hat{A}_k^{\tau'}) \leq \Prob(\hat{A}_k^0)^2 + c'\left(\hat{\Delta} \right)^{a'} \leq \Prob(\hat{A}_k^0)^2 + c'\left(\dfrac{1}{\gamma^{\gamma^k} |\tau-\tau'|}\right)^{a'}\!\!$, 
 \end{center}
 with $a',c'$ independent of $ k, \tau$ and $\tau'$.
\end{lem}
Lemma \ref{hdec} follows from the same arguments used to establish Proposition 3.1 in \cite{OP}.  We presented a modified version of these arguments in the proof of Proposition \ref{dec} above. The $\hat{A}_k^{\tau}$ in Lemma \ref{hdec} only depend on one path in each rectangle, which means we don't have many of the difficulties encountered in Sections  \ref{struct} and \ref{decsec} of this paper.  In fact, the original proof from \cite{OP} goes through without any significant modification.  We will not repeat the proof here, just outline the main steps.  For more details see Section 3 of \cite{OP}.  To prove Lemma \ref{hdec}, we need analogues of Lemmas 3.2-3.5 from \cite{OP}.  Lemma 3.5 (Lemma \ref{3.5} above) is a general result about random walks, and Lemmas 3.2 and 3.4 (which are analogous to Lemma \ref{3.4} above) are just statements about sticking between pairs of paths in the DyDW.  See the discussion following \eqref{B} for an explanation of ``sticking" in the DyDW. Since these lemmas are not specific to the rectangle events they consider, they can be used in their original form.  Lemma 3.3 from \cite{OP} (analogous to Lemma \ref{3} above) is specific to their rectangle events, but the proof only relies on approximating the random walk by a Brownian motion and the boundedness of the density of the normal distribution (see the proof of Lemma \ref{3} for a similar argument).  Combining these lemmas as in \cite{OP} gives the result.  A modified version of this argument was presented above, see the end of Section \ref{decsec}.

Now we prove Proposition \ref{cl1}.  Using the Cauchy-Schwartz inequality as in \eqref{thmp1}-\eqref{thmp2}, we have:
\begin{align}
\Prob (\hat{E}_k \cap [a,b] \neq \emptyset) &\geq 
\Prob \left( \int_a^b \1_{\hat{A}^{\tau}_k} d\tau >0 \right) \nonumber \\
&\geq (b-a)^2\left[ \int_a^b \int_a^b \dfrac{\Prob(\hat{A}^{\tau}_k \cap \hat{A}^{\tau'}_k)}{\Prob(\hat{A}_k)^2} d\tau d\tau' \right]^{-1} \label{h6}
\end{align}
So we need a bound on ${\int_a^b \int_a^b\dfrac{\Prob(\hat{A}^{\tau}_k \cap \hat{A}^{\tau'}_k)}{\Prob(\hat{A}_k)^2} d\tau d\tau'}$ that tends to $(b-a)^2$ as $k \to \infty$.  We split the integral in to two parts:

\begin{enumerate}[(i)]
\item $\{\tau,\tau' \in (a,b) \times (a,b) : \gamma^{\gamma^{k-1}}|\tau-\tau'| \leq 1 \}$
\item $\{\tau,\tau' \in (a,b) \times (a,b) : \gamma^{\gamma^{k-1}}|\tau-\tau'| > 1 \}$
\end{enumerate}
On (i) we use the bound:
\begin{align*}
\dfrac{\Prob(\hat{A}^{\tau}_k \cap \hat{A}^{\tau'}_k)}{\Prob(\hat{A}_k)^2} \leq \dfrac{1}{\Prob(\hat{A}_k)}
\end{align*}
to get:
\begin{align} \label{t2b1}
\int \int_{(i)} \dfrac{\Prob(\hat{A}^{\tau}_k \cap \hat{A}^{\tau'}_k)}{\Prob(\hat{A}_k)^2} d\tau d\tau' \leq \dfrac{1}{\Prob(\hat{A}_k)}\dfrac{2(b-a)}{\gamma^{\gamma^{k-1}}}.
\end{align}
On (ii) we use Lemma \ref{hdec} and $\gamma^{\gamma^{k-1}}|\tau-\tau'| > 1$ to get:
\begin{align*}
\dfrac{\Prob(\hat{A}^{\tau}_k \cap \hat{A}^{\tau'}_k)}{\Prob(\hat{A}_k)^2} \leq 1 + \dfrac{c'}{\gamma^{\gamma^{k}a'}|\tau-\tau'|^{a'} \Prob(\hat{A}_k)^2}\leq 1+ \dfrac{c'}{\gamma^{(\gamma^k-\gamma^{k-1})a'} \Prob(\hat{A}_k)^2}
\end{align*}
So:
\begin{align} \label{t2b2}
\int \int_{(ii)} \dfrac{\Prob(\hat{A}^{\tau}_k \cap \hat{A}^{\tau'}_k)}{\Prob(\hat{A}_k)^2} d\tau d\tau' \leq (b-a)^2\left[1+ \dfrac{c'}{\gamma^{(\gamma^k-\gamma^{k-1})a'} \Prob(\hat{A}_k)^2}\right]
\end{align}

Now, $\Prob(\hat{A}_k)$ is the probability that a random walk started at $-\hat{w}_{k-1}$ exceeds $\hat{w}_k$ after $\hat{d}_k^2$ steps.  If we let $\epsilon_k = \hat{w}_{k-1} / \hat{w}_k$ then this is the same as the probability that a random walk started at $0$ exceeds $(1+\epsilon_k)\hat{w}_k$ after $\hat{d}_k^2$ steps.  We will bound this probability by arguing that the random walk is closely approximated by a Brownian motion and then using standard bounds on the tail of a normal distribution.  To accomplish this we'll use the main result of \cite{OP13}.  For simplicity of notation we'll consider simple symmetric random walks, $S(\cdot)$, started at $x=0$, $t=0$. $S_{\delta}(\cdot)$ will denote the diffusive rescaling of such a path by $\delta$, i.e. $S_{\delta}(t):=S(t/\delta^2)\delta$. The main theorem in \cite{OP13} says that there exists a Brownian motion, $B(\cdot)$, and a sequence of rescaled random walks, $\{ S_{\delta}(\cdot) \}_{\delta > 0}$, such that for any $\alpha < 1/2$ $\Prob(|S_{\delta}-B|_{\infty}>\delta^{\alpha})$ decays faster than any power of $\delta$ (where the $|\cdot|_{\infty}$ norm is restricted to $[0,1]$).  Then by taking $\delta = 1/\hat{d}_k$ and $\alpha=1/3$ we have:
\begin{align}
\Prob(\hat{A}_k) =& \Prob\left(S(\hat{d}_k^2) > (1+\epsilon_k)\hat{w}_k \right)\nonumber\\
\geq& \Prob\left(S_{1/\hat{d}_k}(1) > (1 + \epsilon_k) C\left(\sqrt{\log(\hat{d}_k^2)} + c/\hat{d}_k \right)\right) \nonumber\\
\geq& \Prob \left( B(1) > (1 + \epsilon_k) C\left(\sqrt{\log(\hat{d}_k^2)} + c/\hat{d}_k \right) +(1/\hat{d}_k)^{1/3}\right) \label{Pahat1} \\
& - \Prob\left(|S_{1/\hat{d}_k}- B|_{\infty} > (1/\hat{d}_k)^{1/3} \right). \label{Pahat2} 
\end{align}

Absorb the $1/\hat{d}_k$ terms from \eqref{Pahat1} in to $\epsilon_k$, and use the Theorem from \cite{OP13} to bound \eqref{Pahat2} by $(1/\hat{d}_k)^{1000C^2}$.  This gives, for $k$ sufficiently large:
\begin{align} 
\Prob(\hat{A}_k) \geq & \frac{K}{\sqrt{\log(\hat{d}_k^2)}} \exp \left[-\left((1+\epsilon_k')C\sqrt{\log(\hat{d}_k^2)}\right)^2/2\right] - (1/\hat{d}_k)^{1000C^2} \nonumber \\
\geq& \frac{K'}{\sqrt{\log(\gamma^{2\gamma^k}))}} (\gamma^{\gamma^k})^{-(1+\epsilon_k')^2C^2} \nonumber \\
\geq & K'' (\gamma^{\gamma^k})^{-(1+\epsilon_k'')^2C^2}. \label{supprobbnd}
\end{align}
Now $\epsilon_k'' \to 0$, so for $C$ sufficiently small the right hand side of \eqref{t2b1} converges to $0$ and the right hand side of \eqref{t2b2} converges to $(b-a)^2$ as $k\to \infty$.  This gives the bound needed in \eqref{h6}, which completes the proof of Theorem \ref{thm2}.  

\begin{subsection}{Two-Sided Superdiffusivity}\label{2sidesup} $\quad$

As in the subdiffusive case, one can obtain a two-sided version of this result.  Theorem \ref{thm2sup} states that for $C>0$ sufficiently small,
\begin{align*}
\Prob \left(\exists \tau \in [0,1] \text{ s.t. } \limsup_{t \to \infty} \frac{S_0^{\tau}(t)}{\sqrt{t \log(t)}} \geq C \text{ and } \liminf_{t \to \infty} \frac{S_0^{\tau}(t)}{\sqrt{t \log(t)}} \leq -C\right)=1.
\end{align*}

This follows from a straightforward extension of the results above, combining the reasoning from Sections \ref{struct} and \ref{decsec} with the bounds in this section.  We just need to consider rectangle events with two paths in each rectangle that trap the path from the origin in to a zig-zag pattern.  This is illustrated in figure \ref{fig: 3}. 
\begin{figure}[t!!!]
\begin{center}
 \includegraphics[scale=1]{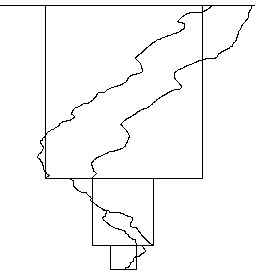}
\end{center}
 \caption{Rough sketch of the first three rectangles and paths for which the $\hat{B}_k^{\tau}$ occur.}
 \label{fig: 3}
\end{figure}

A full proof would involve repeating nearly all the arguments of Sections \ref{struct}-\ref{st2} for a new set of rectangle events.  Instead we give only a quick sketch of the proof and leave the details to the reader.  Use the same definitions for $\hat{R}_k, \hat{t}_k, \hat{l}_k,\hat{r}_k$ as above and let $\hat{l}_k^*,\hat{r}_k^*$ denote the lower left and lower right corners of $\hat{R}_k$, respectively.  Then let $\hat{B}_0^{\tau}=\{S^{\tau}_0(\hat{t}_1) \in [\hat{w}_0,\hat{w}_1]\}$, and for $k\geq 1$:
\begin{align*}
\hat{B}_{2k-1}=& \{S^{\tau}_{\hat{r}_k}(\hat{t}_{2k}) \in [-\hat{w}_{2k},-\hat{w}_{2k-1}],S^{\tau}_{\hat{r}_k^*}(\hat{t}_{2k}) \in [-\hat{w}_{2k},-\hat{w}_{2k-1}]\} \\
\hat{B}_{2k}=& \{S^{\tau}_{\hat{l}_k}(\hat{t}_{2k+1}) \in [\hat{w}_{2k},\hat{w}_{2k+1}],S^{\tau}_{\hat{l}_k^*}(\hat{t}_{2k+1}) \in [\hat{w}_{2k},\hat{w}_{2k+1}]\}
\end{align*}
One can check that if $\bigcap_{k=0}^{\infty} \hat{B}_k^{\tau}$ occurs then $S_0^{\tau}$ will be superdiffusive in both directions, see figure \ref{fig: 3}.

To handle two paths in each rectangle, we then consider analogous rectangle events, $\hat{C}_k^{\tau}$, for a larger system in which paths do not coalesce, as in Section \ref{struct}.  Using the techniques from Sections \ref{struct} and \ref{decsec} we can get a decorrelation bound analogous to Lemma \ref{hdec} above.  For $C$ sufficiently small we can then get an integrable bound on:
\begin{align}\label{2supprod}
\prod_{k=0}^n  \dfrac{\Prob(\hat{C}^{\tau}_k \cap \hat{C}^{\tau'}_k)}{\Prob(\hat{C}_k)^2}
\end{align}
using a similar strategy as in Section \ref{supdim}.  To see this, first notice that:
\begin{align*}
\Prob(\hat{C}^{\tau}_k) \geq& \Prob\left(S_0(\hat{d}_k^2) \in [2\hat{w}_k,\hat{w}_k + \hat{w}_{k+1}] \right)^2 \\
\geq& \hat{K}\Prob\left(S_0(\hat{d}_k^2) \geq 2\hat{w}_k\right)^2 \\
\geq& \hat{K}'(\gamma^{\gamma^k})^{-(1+\hat{\epsilon}_k)^2 8C^2}.
\end{align*}
This follows from the same arguments used to bound $\Prob(\hat{A}_k)$ just before this subsection. 

Now, using the analog of Lemma \ref{hdec} for $\hat{C}^{\tau}_k$ and the same arguments as in the proof of Proposition \ref{supdimprop}, we have:
\begin{align*}
\prod_{k=0}^n  \dfrac{\Prob(\hat{C}^{\tau}_k \cap \hat{C}^{\tau'}_k)}{\Prob(\hat{C}_k)^2} \leq& \hat{K}'' \left( \dfrac{1}{|\tau-\tau'|} \right)^{\frac{16 \gamma^2 (1+\hat{\epsilon}'')^2C^2}{(\gamma-1)}},
\end{align*}
provided that $\hat{a}'(1-\gamma^{-3/2}) - 16C^2 > 0$ ($\hat{a}'$ denotes the value corresponding to $a'$ in the analog of Lemma \ref{hdec}).  This implies that \eqref{2supprod} will be integrable for $C$ sufficiently small.  An application of the Cauchy-Schwartz inequality then implies the existence of two-sided superdiffusive exceptional times, see Section \ref{thm1prf} for a similar argument.

\end{subsection}
\end{section}

\begin{section} {Proof of Theoreom \ref{thm3}}\label{st3}$\quad$
In this section we prove Theorem \ref{thm3}, which says for $C>0$ sufficiently large we have:
\begin{align*}
\Prob \left(\exists \tau \in [0,1] \text{ s.t. } \limsup_{t \to \infty} \frac{S_0^{\tau}(t)}{\sqrt{t \log(t)}} \geq C\right)=0.
\end{align*}
This is a natural counterpoint to Theorem \ref{thm2}, showing that there do not exist exceptional times where the paths are substantially larger.  The proof is inspired by \cite{OP3}, Theorem 8.1.  Let $\gamma>1$ and $d_k=2(\lfloor\frac{\gamma^k}{2}\rfloor+1)$, $t_k = d_0^2 + d_1^2 + ... + d_{k-1}^2$ as in Section \ref{struct}, and let $w_k=2(\lfloor \alpha\sqrt{\log(d_k^2)d_k^2}/2\rfloor+1)$.  We do not use the $\hat{d}_k,\hat{w}_k$ from Section \ref{st2} because we need $t_k$ and $t_{k+1}$ to be of the same order for the following arguments.  We begin by introducing events $\Upsilon^{\tau}_k$:
\begin{align*}
\Upsilon^{\tau}_k := \left\lbrace \underset{t\in [t_{k},t_{k+1}]}{\sup} S_0^{\tau}(t)-S_0^{\tau}(t_{k}) \geq w_k \right\rbrace.
\end{align*}
Suppose that $\Upsilon^{\tau}_k$ occurs for only finitely many $k$.  It then follows that we must have:
\begin{align}\label{t3a}
\limsup_{t \to \infty}\frac{S_0^{\tau}(t)}{\sqrt{t\log(t)}} \leq C,
\end{align}
for some $C>\alpha$.  To see this, assume $\Upsilon^{\tau}_k$ does not occur for $k\geq K$.  Then for $t \in [t_n,t_{n+1}]$ with $n>K$ we have:
\begin{align*}
S_0^{\tau}(t) \leq& (S_0^{\tau}(t_K)-S_0^{\tau}(0)) + (S_0^{\tau}(t_{K+1})-S_0^{\tau}(t_K)) + \dots \\
&+(S_0^{\tau}(t)-S_0^{\tau}(t_n)) \\
\leq&  (S_0^{\tau}(t_K)-S_0^{\tau}(0)) + \sum_{k=K}^{n}w_k \\
\leq&  (S_0^{\tau}(t_K)-S_0^{\tau}(0)) + cw_{n} \\
\leq&  (S_0^{\tau}(t_K)-S_0^{\tau}(0)) + c'\alpha\sqrt{ t_{n+1}\log(t_{n+1})} \\
\leq&  (S_0^{\tau}(t_K)-S_0^{\tau}(0))  + c''\alpha\sqrt{t\log(t)}
\end{align*}
In the last line we used $t\geq t_n \geq (t_n/t_{n+1})t_{n+1}$, and the fact that $t_n/t_{n+1}$ is bounded away from $0$.  \eqref{t3a} then follows easily.

So if for some $\alpha$ we can show $\Prob \left( \exists \tau \in [0,1] \text{ s.t. } \Upsilon^{\tau}_k \right)$ is summable in $k$ we will have:
\begin{align}\label{upeq}
\Prob \left( \exists \tau \in [0,1] \text{ s.t. } \Upsilon^{\tau}_k \text{ occurs i.o.}\right)=0.
\end{align}
Theorem \ref{thm3} then follows from \eqref{t3a}, \eqref{upeq}.  All that is left is to show is:
\begin{clm}
For $\alpha$ sufficiently large $\Prob \left( \exists \tau \in [0,1] \text{ s.t. } \Upsilon^{\tau}_k \right)$ is summable in $k$.
\end{clm}
\begin{proof}
Consider the set of $\tau$ such that $\Upsilon^{\tau}_k$ occurs.  It is a union of disjoint, half-open intervals, let $\upsilon_k$ denote the set of its endpoints in $[0,1]$.  We have:
\begin{align}\label{t3eq}
\Prob \left( \exists \tau \in [0,1] \text{ s.t. } \Upsilon^{\tau}_k \right) &\leq 
\Prob \left( \forall \tau \in [0,1] \; \Upsilon^{\tau}_k \text{ occurs} \right) + 
\Prob \left( \upsilon_k \neq \emptyset  \right).
\end{align}
Now,
\begin{align*}
\Prob \left( \forall \tau \in [0,1] \; \Upsilon^{\tau}_k \text{ occurs} \right) \leq
\Prob \left( \Upsilon^{0}_k  \right).   &
\end{align*}
We can bound $\Prob \left( \Upsilon^{0}_k  \right)$ by approximating the random walk by a Brownian motion.  A similar argument is carried out above to bound $\Prob(\hat{A}_k)$, see the end of Section \ref{st2}. This gives:
\begin{align*}
\Prob \left( \Upsilon^{0}_k  \right) &\leq C_1 \exp\left(-\left(\alpha \sqrt{\log(d_k^2)}\right)^2/2\right)\\
&\leq C_2 d_k^{-\alpha^2}
\end{align*}
which is summable since $d_k \sim \gamma^k$.  So we just need to bound the second term on the right side of \eqref{t3eq}.  Let $\Omega_k$ denote the sites of $\Z^2_{\text{even}}$ which are reachable by $S_0^{\tau}(\cdot)$ for $t \in [0,t_{k+1}]$.  Since the path moves exactly one spatial unit per time unit, $\Omega_k$ is a deterministic triangular region bounded by $(0,0)$, $(-t_{k+1},t_{k+1})$ and $(t_{k+1},t_{k+1})$.  It follows that $|\Omega_k| \leq C_3 t_{k+1}^2 \leq C_4 d_k^4$.  Each endpoint in $\upsilon_k$ comes from an arrow switching at some $x \in \Omega_k$.  So $\upsilon_k =\bigcup_{x \in \Omega_k} \upsilon_k(x)$, where $\upsilon_k(x)$ denotes the endpoints arising from arrow switches at $x$.  Then:
\begin{align*}
\Prob \left( \upsilon_k \neq \emptyset \right) & \leq \E (|\upsilon_k|)\\
&= \sum_{x \in \Omega_k}\E(|\upsilon_k(x)|)\\
&= \sum_{x \in \Omega_k}\E( \text{Number of arrow switches at $x$})\Prob(x \text{ is pivotal for $\Upsilon_k^{\tau}$})\\
&\leq C_5 | \Omega_k | \Prob(\Upsilon_k) \\
&\leq C_6 d_k^4 d_k^{-\alpha^2}.
\end{align*}
A site, $x$, is ``pivotal" for $\Upsilon_k^{\tau}$ if $\Upsilon_k^{\tau}$ occurs for one value of the arrow at $x$, but does not occur for the other value of the arrow at $x$.  The probability that $x$ is pivotal is clearly bounded by twice the probability of the event occurring.  The $\tau$ in the third line is a switching time for $x$, it is omitted in the fourth line because $\tau$ does not affect the probability.  For $\alpha$ sufficiently large $d_k^{4-\alpha^2}$ will be summable, so the proof of the claim, and thus Theorem \ref{thm3}, is complete.  
\end{proof}

\end{section}

\begin{section}{Hausdorff Dimensions of Sets of Exceptional Times} \label{HD}$\quad$
In this section we look at the sets of exceptional times and examine their Hausdorff dimensions.  In Section \ref{tshd} we extend the Hausdorff dimension bounds from \cite{OP} to the sets of two-sided subdiffusive exceptional times and examine the relationship between the dimensions of various related sets of exceptional times.  In Section \ref{supdim} we discuss the sets of superdiffusive exceptional times.  We are able to get a lower bound on the Hausdorff dimension of these sets using similar techniques as in the subdiffusive case.  However, we are not able to get an upper bound at this time.

\begin{subsection}{Two-Sided Subdiffusive Exceptional Times}\label{tshd}

\quad First we look at subdiffusive exceptional times.  We've shown the existence of exceptional times at which $|S_0^{\tau}(t)|$ remains bounded by $j+K\sqrt{t}$ and exceptional times at which $\limsup_{t\rightarrow \infty}|S_0^{\tau}(t)|/\sqrt{t} \leq K$.  The strategy was to show that $S_0^{\tau}$ remained within a stack of diffusively growing rectangles.  The size of the rectangles, and thus the values of $K,j$ in our bounds, were determined by a parameter, $\gamma$.  See the beginning of Section \ref{struct} for the definition of the rectangles, $R_k$, as well as $d_k,t_k$.  This next proposition attempts to capture the relationships between $K,j$ and $\gamma$.

\begin{prop} \label{Kbound}
Let $\sigma_{\gamma}(t)$ denote the right edge of $\bigcup_{k\geq 0}R_k(\gamma)$.  We have:
\begin{align}
\sigma_{\gamma}(t) \leq & 2+ \gamma \sqrt{t} \;\;\;\;\; \text{for all } t\geq 0 \label{Kbound1}, \\
\limsup_{t \rightarrow \infty} \dfrac{\sigma_{\gamma}(t)} {\sqrt{t}}\leq &\sqrt{\gamma^2-1}. \label{Kbound2}
\end{align}
\end{prop}
\begin{proof}
Recall that $d_k=2(\lfloor\frac{\gamma^k}{2}\rfloor+1)$, so $\gamma^k \leq d_k \leq \gamma^k+2$.  This gives:
\begin{align*}
t_k= d_0^2+d_1^2+...+d_{k-1}^2 \geq \gamma^0 +\gamma^2 + ... +\gamma^{2(k-1)} = \dfrac{\gamma^{2k}-1}{\gamma^2-1}.
\end{align*}
Now, for $t_k \leq t < t_{k+1}$, $k\geq1$ we have:
\begin{align}
\sigma_{\gamma}(t) =& d_k \leq \gamma^k +2 \nonumber \\
\leq & \gamma^k \left(\dfrac{\gamma^{2k}-1}{\gamma^2-1}\right)^{\!\!\!-\frac{1}{2}} \!\!\!\!\!\sqrt{t} +2 =\sqrt{\dfrac{\gamma^2-1}{1-\gamma^{-2k}}}\sqrt{t} + 2, \label{Kbound3}
\end{align}
and \eqref{Kbound2} follows immediately.  To see \eqref{Kbound1}, notice that for $t<t_1$ we have $\sigma_{\gamma}(t)=d_0=2 \leq  2+ \gamma \sqrt{t}$.   For $t\geq t_1$, we see $\sqrt{(\gamma^2-1)/(1-\gamma^{-2k})} \leq \gamma$ when $k\geq 1$, so \eqref{Kbound1} follows from \eqref{Kbound3}.
\end{proof}


Now we'd like to consider various sets of exceptional times.  For non-negative $j\in \Z$, we define:
\begin{align*}
T_j^{\pm}(K):=&\{\tau \in [0,\infty): |S_0^{\tau}(t)| \leq j+K\sqrt{t} \;\;\; \forall t\}, \,\; T_{\infty}^{\pm}(K):= \bigcup_{j\geq0} T_j^{\pm}(K), \\
T_j^{+}(K):=&\{\tau \in [0,\infty): S_0^{\tau}(t) \leq j+K\sqrt{t} \;\;\; \forall t\}, \;\;\; \,T_{\infty}^{+}(K):= \bigcup_{j\geq0} T_j^{+}(K), \\
T_j^{-}(K):=&\{\tau \in [0,\infty): S_0^{\tau}(t) \geq -j-K\sqrt{t} \;\;\; \forall t\}, \; T_{\infty}^{-}(K):= \bigcup_{j\geq0} T_j^{-}(K).
\end{align*}
We are interested in the Hausdorff dimensions of these sets.  As in \cite{OP},\cite{OP2}, ergodicity of the DyDW in $\tau$ implies the dimension of any set of exceptional times will be almost surely constant.  Future discussions of such dimensions should be understood to refer to this constant, and thus may only hold almost surely.  In \cite{OP}(Proposition 5.2) it was shown that in the one-sided case, the Hausdorff dimensions of $T_j^+(K)$ and $T_j^{-}(K)$ do not depend on $j\geq0$.  So:
 \begin{center}$\dim_H (T_0^{+}(K))=\dim_H (T_{\infty}^{+}(K))(=\dim_H (T_j^{+}(K)) \text{ for all }j)$, \\
 $\dim_H (T_0^{-}(K))=\dim_H (T_{\infty}^{-}(K))(=\dim_H (T_j^{-}(K)) \text{ for all }j)$.  \end{center}
Modifying their argument, we obtain:
\begin{prop}\label{jdep}
For $T_j^{\pm}(K)$ as defined above, we have:
\begin{align*}
\sup_{K'<K} \dim_H (T_{\infty}^{\pm}(K'))\leq \dim_H (T_1^{\pm}(K))\leq\dim_H (T_{\infty}^{\pm}(K))\leq& \inf_{K''>K}\dim_H (T_1^{\pm}(K'')). 
\end{align*}
\end{prop}
\noindent The reason we take $j=1$ instead of $j=0$ is to prevent the first step of the walk from pushing $|S^{\tau}_0|$ past $j+K\sqrt{t}$ when $K<1$.  If we are only interested in $K\geq 1$ we can take $j=0$ and obtain analogous bounds involving $T_0^{\pm}(K)$.  Notice that $T_1^{\pm}(K)$ and $T_{\infty}^{\pm}(K)$ are increasing functions of $K$, and thus must be continuous for all but countably many $K$.  So for all but countably many $K$, the inequalities from Proposition \ref{jdep} collapse into equalities, and the dimensions will not depend on $j$.  We would conjecture that at least the center inequality should be an equality for all $K$, giving $j$-independence as in the one-sided case, but we are unable to prove this.  

The proof of Proposition \ref{jdep} is motivated by the proof of Proposition 5.2 from \cite{OP}.  The second inequality is trivial; the first and third follow from the same argument.  To see this, pick any $K_1<K_2,j\geq1$ and notice that:
\begin{align*}
\{\tau : |S^{\tau}_{0}(t)| \leq 1 \text{ for all }t\in[0,2n] \} \cap 
\{\tau : |S^{\tau}_{(0,2n)}(t)| \leq j+K_1\sqrt{t-2n} \text{ for all }t\geq 2n \} 
\end{align*}
is contained in $T_1^{\pm}(K_2)$ for $n$ sufficiently large.  This is because $S^{\tau}_{0}(2n)=0$ on the first set, and $j+K_1\sqrt{t-2n}\leq 1+K_2\sqrt{t}$ for large $n$ (this fails for $K_1=K_2$, which is why we don't get full $j$-independence).  The second set is just a translated version of $T_j^{\pm}(K_1)$, and thus has the same Hausdorff dimension.  The first set consists of $\tau$ at which an independent (of $S^{\tau}_{(0,2n)}$) event of positive probability occurs.  Thus, by the same ergodicity arguments used in Proposition 5.2 from \cite{OP}, intersection with the first set does not decrease the dimension.  So:
\begin{align*}
\dim_H (T_j^{\pm}(K_1))\leq\dim_H (T_1^{\pm}(K_2)) \text{   for all }j\geq 1, K_1<K_2,
\end{align*}
which proves both the first and third inequalities.

Now we focus on comparing the Hausdorff dimensions of $T_{\infty}^{\pm}(K)$ and other, related sets of exceptional times.  We may drop $j$ from the notation, and when $j$ is not specified it should be understood that we are discussing $T_{\infty}^{\pm}(K)$ (i.e., $T^{\pm}(K):=T_{\infty}^{\pm}(K)$).  Proposition \ref{jdep} allows us to translate the coming bounds into bounds for $\dim_H (T_1^{\pm}(K))$ (or $\dim_H (T_0^{\pm}(K))$ for $K\geq1$).

We now consider dynamical times at which $S^{\tau}_0(t)$ displays exceptional behavior as $t$ goes to $\infty$.  That is, we look at times at which \textit{$S^{\tau}_0$ is $K$-subdiffusive in the limit:}
\begin{align*}
\tilde{T}^{\pm}(K):= \{\tau\in[0,\infty) : \limsup_{t\rightarrow\infty}|S_0^{\tau}(t)|/\sqrt{t}\leq K\}.
\end{align*}\vspace{-2mm}
We'd like to relate this set to $T^{\pm}(K)$.  Notice that:
\begin{align}
T^{\pm}(K):=& \{\tau\in[0,\infty) :\exists j\text{ s.t. }|S_0^{\tau}(t)|\leq j+K\sqrt{t}\;\;\;\forall t\} \nonumber\\
=& \{\tau\in[0,\infty) :\exists N,j\text{ s.t. }|S_0^{\tau}(t)|\leq j+K\sqrt{t}\;\;\;\forall t\geq N\}.\label{Tunion}
\end{align}
This is because the first set is clearly contained in the second, and for any $\tau$ in the second set we can simply choose a larger value for $j$ to make the inequality $|S_0^{\tau}(t)|\leq j+K\sqrt{t}$ hold for all $t$.  This implies:
\begin{align*}\vspace{-4mm}
\tilde{T}^{\pm}(K) = \bigcap_{K'>K}T^{\pm}(K'),
\end{align*}\vspace{-2mm}
so:\vspace{-2mm}
\begin{align}\vspace{-2mm} \label{tildehat}
\dim_H(\tilde{T}^{\pm}(K)) = \inf_{K'>K}\dim_H(T^{\pm}(K')).
\end{align}
As in the discussion following Proposition \ref{jdep}, monotonicity in $K$ implies that \eqref{tildehat} will also equal $\dim_H(T^{\pm}(K))$ except for at most countably many $K$.  It may be that there is equality for all $K$, but this does not follow from our arguments. 

In \cite{OP} it was shown that almost surely, for all $\tau$, all walks in the DyDW are recurrent, and all pairs of walks coalesce (see Theorem 2.1, Remark 2.3 from \cite{OP}).  This implies:
\begin{align}
T^{\pm}(K)&=\{\tau\in[0,\infty) :\exists N,j\text{ s.t. }|S_0^{\tau}(t)|\leq j+K\sqrt{t}\;\;\;\forall t\geq N\} \;\;\; \text{(by \eqref{Tunion})}\nonumber\\ 
&\stackrel{\text{a.s.}}{=} \bigcup_{n\geq 0} \{\tau\in[0,\infty) :\exists j \text{ s.t. }|S_{(0,2n)}^{\tau}(t)|\leq j+K\sqrt{t}\;\;\;\forall t\geq 2n\}. \label{thateq}
\end{align}
To see this, notice that on the second set, $S_0^{\tau}$ will a.s. eventually coalesce with $S_{(0,2n)}^{\tau}$, so for $t$ large we will have $|S_0^{\tau}(t)| =|S_{(0,2n)}^{\tau}(t)| \leq j+K\sqrt{t}$.  For $\tau$ in the first set, let $N^*$ be the first time $S_0^{\tau}$ returns to zero after $N$.  Then $|S_{(0,N^*)}(t)| \leq j+K\sqrt{t}\;\;\;\forall t\geq N^*$.  This proves \eqref{thateq}.

Using \eqref{thateq} and the recurrence of all paths, we also get:
\begin{align*}
T^{\pm}(K)\stackrel{\text{a.s.}}{=} \bigcap_{m\geq 0}\bigcup_{n\geq m} \{\tau\in[0,\infty) :\exists j \text{ s.t. }|S_{(0,2n)}^{\tau}(t)|\leq j+K\sqrt{t}\;\;\;\forall t\geq 2n\},
\end{align*}
which is a tail random variable with respect to the underlying $\xi_{(x,t)}^{\tau}$ processes.  Similar reasoning also applies to $\tilde{T}^{\pm}(K)$.  These observations imply:
\begin{align} \label{denseeq}
\Prob(T^{\pm}(K)\cap [0,\epsilon]= \emptyset)= 0 \text{ or } 1 \;\;\;\;\;\text{for all } K>0,\epsilon\geq 0, \nonumber\\
\Prob(\tilde{T}^{\pm}(K)\cap [0,\epsilon]= \emptyset)= 0 \text{ or } 1 \;\;\;\;\;\text{for all } K>0,\epsilon\geq 0.
\end{align}
An easy consequence of \eqref{denseeq} is given by the following proposition:
\begin{prop}
Almost surely, for every $K>0$, $\tilde{T}^{\pm}(K)$, $T^{\pm}(K)$ will each be either empty, or dense in $[0,\infty)$.
\end{prop}

Now we prove a lower bound for the Hausdorff dimension of $\tilde{T}^{\pm}(K)$.  The above results (Proposition \ref{jdep}, \eqref{tildehat}) allow us to translate the following bound into lower bounds for $\dim_H (T^{\pm}(K))$ and $\dim_H (T_j^{\pm}(K))$.  In fact, our lower bound is continuous in $K$, so we get the same bound for all these sets of exceptional times.  Now, let $\tilde{\gamma}(K):= \sqrt{K^2+1}$, so that:
\begin{align} \label{sigmaK}
\limsup_{t \rightarrow \infty} \dfrac{\sigma_{\tilde{\gamma}(K)}(t)} {\sqrt{t}}\leq K
\end{align}
(see Proposition \ref{Kbound}).  Given $\gamma$, let $C_{\infty}(\gamma)$ be the corresponding rectangle event for Brownian motions; that is, the event that two independent Brownian motions started at $\pm \gamma^{-1}$ stay within $[-1,1]$ for $0\leq t \leq 1$.  Then we have:
\begin{prop}\label{ldim}
\begin{align}
\dim_H(\tilde{T}^{\pm}(K)) \geq 1 - \dfrac{\log \Prob(C_{\infty}(\tilde{\gamma}(K)))^{-1}}{\log \tilde{\gamma}(K)} =: 1-b_{\infty}(K). \label{dimbnd}
\end{align}
\end{prop}
As an immediate consequence of this we have $\dim_H(\tilde{T}^{\pm}(K)) \to 1$ as $K \to \infty$.  Proposition \ref{ldim} is established by a modification of the arguments used in Proposition 5.3 from \cite{OP}.  We will drop the $K$ dependence from the notation for the moment.  First we define a family of random measures, $\sigma_{n,m}$, that play the role of the $\sigma_n$ from their proof.  As above, we take $C_k:=C_k^0$.  Given a Borel set $E$ in $[0,1]$, $n \geq m$, we define:
\begin{align*}
\sigma_{n,m}(E) := \int_E \prod_{k=m}^n \dfrac{\1_{C_k^{\tau}}}{\Prob(C_k)} d\tau,
\end{align*}
and notice that $\sigma_{n,m}$ is supported on $\bar{E}_{n,m}$, the closure of:
\begin{align}
E_{n,m}:= \{ \tau \in [0,1] : \bigcap_{k=m}^n C_k^{\tau}\; \text{  occurs}\}.
\end{align}
Now, reasoning as in \cite{OP}, we would like to show $\Prob(\sigma_{n,m}([0,1])>1/2)>c$ (for some $c$ independent of $n$) and bound the expectation of the $\alpha$-energy of $\sigma_{n,m}$, defined to be:
\begin{align*}
\int_0^1 \int_0^1 \frac{1}{|\tau-\tau'|^{\alpha}} d\sigma_{n,m}(\tau) d\sigma_{n,m}(\tau')
\end{align*}
Similarly to \eqref{thmp1}-\eqref{thmp3} above, we have for all $n\geq m$:
\begin{align*}
\E [\sigma_{n,m}([0,1])^2] =& \int_0^1 \int_0^1 \prod_{k=m}^n \dfrac{\Prob(C^{\tau}_k \cap C^{\tau'}_k)}{\Prob(C_k)^2} d\tau d\tau' \\
\leq & c \int_0^1 \int_0^1 \dfrac{1}{|\tau-\tau|^{b_m}} d\tau d\tau',
\end{align*}
with:
\begin{align*}
b_m = \log [ \sup_{k \geq m} (\Prob ( C_k)^{-1})]/ \log \tilde{\gamma}.
\end{align*}

For $b_m<1$, an application of the Cauchy-Schwartz inequality as in \eqref{thmp1}-\eqref{thmp3} gives:
\begin{align*}
\Prob\left[\sigma_{n,m}([0,1])>1/2\right] \geq& \dfrac{\left( \E \left[\sigma_{n,m}([0,1]) \1_{\sigma_{n,m}([0,1])>1/2} \right]\right)^2}{\E \left[\sigma_{n,m}([0,1])^2\right]} \\
\geq& \dfrac{\left( \E \left[\sigma_{n,m}([0,1]) \right] - 1/2 \right)^2}{\E \left[\sigma_{n,m}([0,1])^2\right]} \\
=& \dfrac{\left( 1 - 1/2 \right)^2}{\E \left[\sigma_{n,m}([0,1])^2\right]} \\
\geq& c' \left( \int_0^1 \int_0^1 \dfrac{1}{|\tau-\tau|^{b_m}} d\tau d\tau' \right)^{-1}.
\end{align*}
So $\Prob(\sigma_{n,m}([0,1])>1/2)>c''$, where $c''$ doesn't depend on $n$.

We can bound the expectation of the $\alpha$-energy of $\sigma_{n,m}$ by:
\begin{align*}
\E  \left[ \int_0^1 \int_0^1 \frac{1}{|\tau-\tau'|^{\alpha}} d\sigma_{n,m}(\tau) d\sigma_{n,m}(\tau') \right] =& \int_0^1 \int_0^1 \frac{1}{|\tau-\tau'|^{\alpha}}\prod_{k=m}^n \dfrac{\Prob(C^{\tau}_k \cap C^{\tau'}_k)}{\Prob(C_k)^2} d\tau d\tau' \\
\leq& \int_0^1 \int_0^1 \frac{1}{|\tau-\tau'|^{\alpha+b_m}} d\tau d\tau'
\end{align*}
We assume that $K$ is large enough to make the right-hand side of \eqref{dimbnd} positive (otherwise there is nothing to prove).  Then $b_{\infty} < 1$ and by the diffusive scaling of the events $C_k$ we have $b_m \rightarrow b_{\infty} <1$ as $m\rightarrow \infty$.  So given any $\alpha < 1- b_{\infty}$, there exists $m$ large enough such that $\alpha + b_m <1$.  This gives a uniform (in $n$) bound on the expectation of the $\alpha$-energy of $\sigma_{n,m}$.  Arguing as in Proposition 5.3 of \cite{OP}, we can then use the extension of Frostman's lemma from \cite{OP3} to conclude that:
\begin{align}
\dim_H ( \bigcap_{n\geq m} \bar{E}_{n,m}) \geq \alpha\; \text{ with positive probability.}\label{dimbnd1}
\end{align}
Now, for our chosen $K$, $\bigcap_{n\geq m} E_{n,m} \subset \tilde{T}^{\pm}(K)$ for all $m$ (using \eqref{sigmaK} and the a.s. coalescence of all paths).  We've shown that given any $\alpha < 1- b_{\infty}$, \eqref{dimbnd1} holds for some sufficiently large $m$.  Also, $\bigcap_{n\geq m} E_{n,m}=\bigcap_{n\geq m} \bar{E}_{n,m}$ almost surely, by the same argument used to establish \eqref{thmp5}.  Combining these observations, we have:
\begin{align*}
\dim_H(\tilde{T}^{\pm}(K)) \geq  1-b_{\infty}(K),
\end{align*}
(almost surely by ergodicity in $\tau$ of the DyDW).  This proves Proposition \ref{ldim}.

\begin{rmk}
One may wish to consider ``asymmetrical" exceptional times.  That is, exceptional times where the $K$ of $\tilde{T}^{\pm}(K)$, $T^{\pm}(K)$, $T_j^{\pm}(K)$, etc. is replaced by two constants, $K_L$,  and $K_R$, giving different bounds on the left and right sides.  One can obtain an analogous lower bound for the dimension of these asymmetrical exceptional times using the ``skewed rectangle" construction described in Remark \ref{skew}.  
\end{rmk}

Now we look at upper bounds for the Hausdorff dimension of the sets of two-sided exceptional times.  This is a straightforward extension of the results in Section 5.2 of \cite{OP}.  Following \cite{OP}, we state the results for the asymmetrical case.   So we give an upper bound for $\dim_H( T^{-}_1(K_L)\cap T_1^+(K_R))$.  Recall:
\begin{align}
T^{-}_1(K_L)\cap T_1^+(K_R)= \{ \tau \in [0,\infty) : -1-K_L\sqrt{t} \leq S_0^{\tau}(t) \leq 1+K_R\sqrt{t} \;\; \text{for all $t$}\},
\end{align}
using the definitions given earlier in this section.

Proposition 5.5 from \cite{OP} gives the bound $\dim_H( T^{-}_1(K)) \leq 1- p(K)$, where $p(K)\in (0,1)$ is the solution to:
\begin{align}
f(p,K):= \dfrac{ \sin(\pi p/2) \Gamma(1+p/2)}{\pi} \sum_{n=1}^{\infty} \dfrac{(\sqrt{2}K)^n}{n!} \Gamma((n-p)/2)=1.
\end{align}
They also prove that $T^{-}_1(K_L)\cap T_1^+(K_R)$ is empty when $p(K_L)+p(K_R)>1$ (see \cite{OP}, Proposition 5.8).  The function $p(K)$ comes from \cite{OP18}, where it is shown that $p(K)$ is continuous and decreasing on $(0,\infty)$, tending to $0$ as $K$ goes $\infty$, tending to $1$ as $K$ goes to $0$.

The upper bound from \cite{OP} is established by partitioning $[0,1]$ into intervals of equal length, and estimating the number of these needed to cover $T^{-}_1(K)$.  An application of the FKG inequality, as in Proposition 5.8 of \cite{OP}, extends the bound to the two-sided case, giving:

\begin{prop}
\begin{align*}
\dim_H( T^{-}_1(K_L)\cap T_1^+(K_R)) \leq 1- p(K_L)-p(K_R),
\end{align*}
so:
\begin{align*}
\dim_H( T^{\pm}_1(K)) \leq 1- 2p(K).
\end{align*}
\end{prop}

Note that, as with the lower bound, continuity of the bound combined with our previous results gives an identical bound for $\dim_H (T^{\pm}(K))$, $\dim_H (T_j^{\pm}(K))$, $\dim_H(\tilde{T}^{\pm}(K))$ and their asymmetrical analogues.
\end{subsection}

\begin{subsection}{Superdiffusive Exceptional Times}\label{supdim}

\quad In this section we will derive a lower bound for the Hausdorff dimension of the superdiffusive exceptional times.  We consider the rectangle events $\hat{A}_k^{\tau}$ from Section \ref{st2}.  The proof relies on the same techniques as in Proposition 5.3 from \cite{OP} and Proposition \ref{ldim} above.  As in Section \ref{st2} of this paper, the proof is complicated by the fact that $\Prob(\hat{A}_k) \to 0$.  We will give a bound for the set of $\tau$ such that $\hat{A}_k^{\tau}$ occurs for all $k \geq 0$.  This set is a subset of the set of the superdiffusive exceptional times, so a lower bound on its dimension gives the bound we need.  To obtain an upper bound on the dimension of the superdiffusive exceptional times we would instead need to consider a (possibly) larger set.  Attempts in this direction have not been succesful, so we will only give a lower bound.

Let $\hat{T}^+(C),\hat{T}^-(C)$ denote the sets of $C$-superdiffusive times, i.e.:
\begin{align*}
\hat{T}^+(C)=\{ \tau \in[0,\infty) &:\limsup_{t \to \infty} S_0^{\tau}(t)/\sqrt{t\log(t)} \geq C \}\\
\hat{T}^-(C)=\{ \tau \in[0,\infty) &:\liminf_{t \to \infty} S_0^{\tau}(t)/\sqrt{t\log(t)} \leq -C \}.
\end{align*}
These two sets will have the same Hausdorff dimension due to symmetry.  So we will focus on $\hat{T}^+(C)$.  Taking $a'$ to be the value given by Lemma \ref{hdec} we have the following proposition:

\begin{prop}\label{supdimprop}
For $\gamma,C$ such that $a'(1-\gamma^{-3/2}) - 2C^2 > 0$ we have:
\begin{align}
\dim_H(\hat{T}^{+}(C)) \geq 1 - \frac{2 \gamma^2 C^2}{(\gamma-1)}. \label{supdimpropbnd}
\end{align}
\end{prop}

As an immediate consequence of Proposition \ref{supdimprop} we see that $\dim_H(\hat{T}^{+}(C)) \to 1$ as $C \to 0$.  The remainder of this section will be devoted to the proof of Proposition \ref{supdimprop}.  Similar to Proposition \ref{ldim} above, we consider the measures:
\begin{align*}
\hat{\sigma}_{n,m}(E) := \int_E \prod_{k=m}^n \dfrac{\1_{\hat{A}_k^{\tau}}}{\Prob(\hat{A}_k)} d\tau,
\end{align*}
which are supported on $\bar{\hat{E}}_{n,m}$, the closure of:
\begin{align}
\hat{E}_{n,m}:= \{ \tau \in [0,1] : \bigcap_{k=m}^n \hat{A}_k^{\tau}\; \text{  occurs}\}.
\end{align}
For the sake of simplicity we will only consider the case $m=0$ and let $\hat{\sigma}_{n}=\hat{\sigma}_{n,0}$, $\hat{E}_{n}=\hat{E}_{n,0}$.  For $n\geq 0$ we want to bound $\E [\hat{\sigma}_{n}([0,1])^2]$ and the expectation of the $\alpha$-energy of $\hat{\sigma}_{n}$:
\begin{align}\label{supbnd}
\E  \left[ \int_0^1 \int_0^1 \frac{1}{|\tau-\tau'|^{\alpha}} d\hat{\sigma}_{n}(\tau) d\hat{\sigma}_{n}(\tau') \right] =& \int_0^1 \int_0^1 \frac{1}{|\tau-\tau'|^{\alpha}} \prod_{k=0}^n  \dfrac{\Prob(\hat{A}^{\tau}_k \cap \hat{A}^{\tau'}_k)}{\Prob(\hat{A}_k)^2} d\tau d\tau'.
\end{align}
$\E [\hat{\sigma}_{n}([0,1])^2]$ corresponds to $\alpha=0$ in \eqref{supbnd}, so if we can obtain an integrable bound for \eqref{supbnd} with $\alpha>0$ we will have $\dim_H(\hat{T}(C))\geq\alpha$.  This follows from Cauchy-Schwarz and Frostman's lemma, as in Proposition 5.3 of \cite{OP} and Proposition \ref{ldim} above.  We will use Lemma \ref{hdec} from Section \ref{st2} to bound the product on the right side of \eqref{supbnd}.  Lemma \ref{hdec} gives $c',a'\in (0,\infty)$ independent of $ k, \tau$ and $\tau'$ such that:
 \begin{center}
$\dfrac{\Prob(\hat{A}_k^{\tau} \cap \hat{A}_k^{\tau'})}{\Prob(\hat{A}_k)^2} \leq  1 + \dfrac{c'}{\Prob(\hat{A}_k)^2}\left(\dfrac{1}{\gamma^{\gamma^k} |\tau-\tau'|}\right)^{a'}\!\!$. 
 \end{center}
Let $N_0= \lfloor \log_{\gamma}(2\log_{\gamma}(1/|\tau-\tau'|)) \rfloor + 1$, so that $\gamma^{\gamma^{N_0}/2}|\tau-\tau'| \geq 1$.  We will split the product at $N_0$ to obtain our bound.  First consider $k>N_0$:
\begin{align}
\prod_{k=N_0+1}^n \dfrac{\Prob(\hat{A}^{\tau}_k \cap \hat{A}^{\tau'}_k)}{\Prob(\hat{A}_k)^2} \leq &
\prod_{k=N_0+1}^{\infty} 1 + \dfrac{c'}{\Prob(\hat{A}_k)^2}\left(\dfrac{1}{\gamma^{\gamma^k} |\tau-\tau'|}\right)^{a'} \nonumber \\
 = & \prod_{k=N_0+1}^{\infty} 1 + \dfrac{c'}{\Prob(\hat{A}_k)^2}\left(\dfrac{1}{\gamma^{\gamma^k-\gamma^{N_0/2}}( \gamma^{\gamma^{N_0/2}}|\tau-\tau'|)}\right)^{a'} \nonumber \\
\leq & \prod_{k=N_0+1}^{\infty} 1 + \dfrac{c'}{\Prob(\hat{A}_k)^2}\left(\dfrac{1}{\gamma^{\gamma^k-\gamma^{N_0/2}}}\right)^{a'} \nonumber \\
\leq & \prod_{k=N_0+1}^{\infty} 1 + \dfrac{c'}{\Prob(\hat{A}_k)^2}\left(\dfrac{1}{\gamma^{\gamma^k(1-\gamma^{-3/2})}}\right)^{a'} \nonumber \\
\leq & \ \  \prod_{k=1}^{\infty} \ \ 1 + \dfrac{c'}{\Prob(\hat{A}_k)^2}\left(\dfrac{1}{\gamma^{\gamma^k(1-\gamma^{-3/2})}}\right)^{a'} \label{supdimeq}
\end{align}
In the second to last step we used $N_0\geq 1$, $k\geq N_0 +1$.  In Section \ref{st2} (see \eqref{supprobbnd}) we saw that:
\begin{align*}
\Prob(\hat{A}_k) \geq &  K'' (\gamma^{\gamma^k})^{-(1+\epsilon_k'')^2C^2},
\end{align*}
where $\epsilon_k'' \to 0$.  This bound was proven for $k$ sufficiently large, but we can make it hold for all $k$ by decreasing $K''$.  So provided that $a'(1-\gamma^{-3/2}) - 2C^2 > 0$ \eqref{supdimeq} gives:
\begin{align}
\prod_{k=N_0+1}^n \dfrac{\Prob(\hat{A}^{\tau}_k \cap \hat{A}^{\tau'}_k)}{\Prob(\hat{A}_k)^2} \leq & K(C). \label{supdimbnda}
\end{align}

Now we consider the $k\leq N_0$ terms.  Given any fixed $\epsilon''>0$ we can make $\epsilon_k'' < \epsilon''$ for all $k$ by decreasing $K''$,  since $\epsilon_k'' \to 0$.  Similarly to \eqref{Pahat1}-\eqref{supprobbnd} above we will absorb lower order terms into $K''$ and $\epsilon ''$.  To keep the notation simple we will continue to use $K'',\epsilon''$ to denote these updated values.  So for $n \leq N_0$, we have:
\begin{align*}
\prod_{k=0}^n \dfrac{\Prob(\hat{A}^{\tau}_k \cap \hat{A}^{\tau'}_k)}{\Prob(\hat{A}_k)^2} \leq & \prod_{k=0}^{N_0} \dfrac{1}{\Prob(\hat{A}_k)} \\
\leq & \prod_{k=0}^{N_0} \dfrac{1}{K'' (\gamma^{\gamma^k})^{-(1+\epsilon_k'')^2C^2}} \\
\leq & \dfrac{1}{{(K'')}^{N_0+1}\gamma^{-(1+\epsilon'')^2C^2\sum_{k=0}^{N_0} \gamma^k}} \\
\leq & \dfrac{(\gamma^{\gamma^{N_0+1}})^{(1+\epsilon'')^2C^2/(\gamma-1)}}{{(K'')}^{N_0+1}} \\
\leq & K'' \left( \dfrac{1}{|\tau-\tau'|} \right)^{\frac{2 \gamma^2 (1+\epsilon'')^2C^2}{(\gamma-1)}}.
\end{align*}
Combined with \eqref{supdimbnda} this gives the bound needed in \eqref{supbnd} which completes the proof.

One could derive a similar lower bound for the dimension of the two-sided superdiffusive times, $\hat{T}^{\pm}(C) = \hat{T}^+(C) \cap \hat{T}^-(C)$, using the same techniques combined with the ideas of Section \ref{2sidesup}.  We will not present a proof of such a result.

\end{subsection}

\vspace*{1in}
\noindent \textbf{Acknowledgements:}  The research reported in this paper was supported in part by NSF grants OISE-0730136, DMS-1007524 and DMS-1207678.  I would like to thank Chuck Newman, my advisor, for introducing me to the problem and our many helpful discussions.  I would also like to thank Behzad Mehrdad and Arjun Krishnan for a helpful discussion related to the proof of Theorem \ref{thm2}.
\end{section}

\end{document}